\UseRawInputEncoding
\documentclass[12pt, reqno]{amsart}
\usepackage{amssymb,latexsym,amsmath,amscd,amsfonts}
\usepackage{latexsym}
\usepackage[mathscr]{eucal}
\usepackage{bm}
\usepackage{mathptmx}

9.5in \numberwithin{equation}{section}

\newcommand{\beg}{\begin{equation}}
\newcommand{\eeg}{\end{equation}}
\newcommand{\ben}{\begin{eqnarray*}}
	\newcommand{\een}{\end{eqnarray*}}

\newtheorem{thm}{Theorem}[section]
\newtheorem{cor}[thm]{Corollary}
\newtheorem{lem}[thm]{Lemma}

\newtheorem{prop}[thm]{Proposition}
\numberwithin{equation}{section} 
\theoremstyle{definition}
\newtheorem{defn}[thm]{Definition}

\newtheorem{eg}[thm]{Example}

\newcommand{\HS}{\mathcal H}

\newcommand{\D}{\mathbb{D}}
\newcommand{\T}{\mathbb{T}}
\newcommand{\ft}{\mathcal F_O}

\newcommand{\gn}{\mathbb{G}_n}

\newcommand{\M}{\mathcal{M}}

\newcommand{\ov}{\overline}

\begin{document}
\title[The c.n.u. $\Gamma_n$-contractions]
{A Nagy-Foias program for a c.n.u. $\Gamma_n$-contraction}

\author[Bappa Bisai]{Bappa Bisai}
\address[Bappa Bisai]{Stat-Math Unit, Indian Statistical Institute, 203 B.T Road, Baranagar, Kolkata-700106, India.} \email{bappa.bisai1234@gmail.com}

\author[Sourav Pal]{Sourav Pal}
\address[Sourav Pal]{Mathematics Department, Indian Institute of Technology Bombay,
Powai, Mumbai - 400076, India.} \email{sourav@math.iitb.ac.in}



\thanks{The first named author has been supported by the Ph.D fellowship of the University Grants Commissoin (UGC) and the Visiting Scientist Fellowship of Indian Statistical Institute, Kolktata. The second named author
is supported by the Seed Grant of IIT Bombay, the CPDA and the
MATRICS Award (Award No. MTR/2019/001010) of
Science and engineering Research Board (SERB), India.}

\begin{abstract}

A tuple of commuting Hilbert space operators $(S_1, \dots, S_{n-1}, P)$ having the closed symmetrized polydisc 
\[
\Gamma_n = \left\{ \left(\sum_{i=1}^{n}z_i, \sum\limits_{1\leq i<j\leq n} z_iz_j, \cdots, \prod_{i=1}^{n}z_i\right) : |z_i|\leq 1\,, \; \; \; 1\leq i \leq n-1 \right\}
\]
as a spectral set is called a $\Gamma_n$-contraction. From the literature we have that a point $(s_1, \dots , s_{n-1},p)$ in $\Gamma_n$ can be represented as $s_i=c_i+pc_{n-i}$ for some $(c_1, \dots, c_{n-1}) \in \Gamma_{n-1}$. We construct a minimal $\Gamma_n$-isometric dilation for a particular class of c.n.u. $\Gamma_n$-contractions $(S_1, \cdots, S_{n-1},P)$ and obtain a functional model for them. With the help of this model we express each $S_i$ as $S_i=C_i+PC_{n-i}$, which is an operator theoretic analogue of the scalar result. We also produce an abstract model for a different class of c.n.u. $\Gamma_n$-contractions satisfying $S_i^*P=PS_i^*$ for each $i$. By exhibiting a counter example we show that such abstract model may not exist if we drop the hypothesis that $S_i^*P=PS_i^*$. We apply this abstract model to achieve a complete unitary invariant for such c.n.u. $\Gamma_n$-contractions. Additionally, we present different necessary conditions for dilation and a sufficient condition under which a commuting tuple $(S_1, \dots , S_{n-1},P)$ becomes a $\Gamma_n$-contraction. The entire program goes parallel to the operator theoretic program developed by Sz.-Nagy and Foias for a c.n.u. contraction.

\end{abstract}

\maketitle


\section{Introduction and preliminaries}

\vspace{0.4cm}

\noindent This paper is a sequel of \cite{S:P2} and \cite{sourav14}. Throughout the paper, we consider only bounded operators acting on complex Hilbert spaces. A \textit{contraction} is an operator with norm not greater than $1$ and a \textit{completely non-unitary} (c.n.u.) contraction is a contraction which does not act like a unitary on any of its nonzero reducing subspaces.

A few decades ago, Sz.-Nagy and Foias initiated and carried out an operator theoretic program for a completely non-unitary (c.n.u.) contraction. Explicit construction of dilation, functional model and characteristic function are the principal parts of that campaign. The aim of this article is to develop an analogous theory for a class of c.n.u. operator tuples associated with the symmetrized polydisc. For $n\geq 2$, the symmetrized polydisc $\mathbb G_n$, where
\[
\mathbb G_n=\left\{\left(\sum\limits_{1\leq i\leq n}z_i, \sum\limits_{1\leq i<j\leq n}z_iz_j, \dots, \prod_{i=1}^nz_i\right): |z_i|<1, i=1, \dots, n\right\},
\] 
is a family of polynomially convex but non-convex domains which naturally arise in the spectral interpolation problem. If $\mathcal M_n(\mathbb C)$ is the space of $n\times n$ complex matrices and if $\mathcal B_1^n$ is its spectral unit ball, then $A\in \mathcal B_1^n$ (that is, the spectral radius $r(A)<1$) if and only if $\pi_n(\sigma(A)) \in \mathbb G_n$, where $\sigma(A)$ is the spectrum of $A$ and $\pi_n:\mathbb C^n \rightarrow \mathbb C^n$ is the symmetrization map
\[
\pi_n(z_1, \dots , z_n)=\left(\sum_{1\leq i\leq n} z_i,\sum_{1\leq
i<j\leq n}z_iz_j,\dots, \prod_{i=1}^n z_i \right).
\]
Note that $\mathbb G_1$ is the unit disc $\D$ and for $n\geq 2$, $\mathbb G_n$ is the symmetrization of the points in the polydisc $\D^n$, i.e. $\gn=\pi_n(\D^n)$. Here we shall consider $\gn$ for $n\geq 2$ only. Naturally a bounded domain like $\mathbb G_n$ that has complex-dimension $n$, is much easier to deal with than an unbounded $n^2$-dimensional object like $\mathcal B_1^n$. Apart from its rich function theoretic and geometric aspects (e.g. \cite{costara1, NiPa, Pal-Roy1, B-B-M-1, edi-zwo, NiPfZw, N-P-T-Zw1, Pal-Roy2, Su-T-Tu}), the symmetrized polydisc has evloved independently as an object of operator theory in past two decades, \cite{ay-jfa, ay-ieot, ay-jot, tirtha-sourav, tirtha-sourav1, S:B, J:Sarkar, S:P, B:P01, B:Pratfun, B:P, S:P2, A:Pal, sourav14}.

\begin{defn}\label{def1}
	A commuting tuple of Hilbert space operators $(S_1, \dots, S_{n-1}, P)$ is called a $\Gamma_n$-\textit{contraction} if the closed symmetrized polydisc $\Gamma_n$, where
	\[
\Gamma_n=\ov{\gn}=\left\{ \left( \sum\limits_{1\leq i\leq n}z_i, \sum\limits_{1\leq i<j\leq n}z_iz_j, \dots, \prod_{i=1}^nz_i\right): |z_i|\leq 1, i=1, \dots, n \right\},
\]
is a spectral set for $(S_1, \dots, S_{n-1}, P)$ i.e. the Taylor joint spectrum $\sigma_T(S_1,\dots, S_{n-1}, P)\subseteq \Gamma_n$ and 
	\[
	\|r(S_1, \dots, S_{n-1}, P)\| \leq \|r\|_{\infty, \Gamma_n} = \text{sup}\{|r(z_1,\dots,z_n)| : (z_1,\dots,z_n)\in \Gamma_n\},
	\]
	for all rational functions $r=p/q$ such that $q$ does not have any zero in $\Gamma_n$.
\end{defn}
Needless to mention that if $(S_1, \dots , S_{n-1},P)$ is a $\Gamma_n$-contraction, then so is $(S_1^*, \dots , S_{n-1}^*,P^*)$ and $P$ is a contraction. It was proved in \cite{S:P2} that to every $\Gamma_n$-contraction $(S_1, \dots, S_{n-1},P)$, there are unique operators $A_1, \dots, A_{n-1}\in \mathcal{B}(\mathcal{D}_P)$ such that 
\[
S_i - S_{n-i}^*P = D_PA_iD_P \,, \qquad i=1, \dots, n-1.
\] 
The unique operator tuple $(A_1, \dots, A_{n-1})$, which was named the \textit{fundamental operator tuple} or in short $\ft$-\textit{tuple} of $(S_1, \dots, S_{n-1},P)$, plays central role in every section of operator theory on the symmetrized polydisc, (e.g. \cite{S:P, S:P2, B:P01, B:Pratfun, B:P, S:P1, sourav14}).

Unitaries, isometries and co-isometries are special classes of contractions. There are natural analogues of these classes for $\Gamma_n$-contractions (see Section \ref{commodel}). The following theorem provides a more explicit description of $\Gamma_n$-unitaries and $\Gamma_n$-isometries.

\begin{thm}[\cite{sourav14}, Theorems 4.2 \& 4.4]\label{charofgamuniiso}
	A commuting tuple of operators $(S_1, \dots, S_{n-1},P)$ is a $\Gamma_n$-unitary (or, a $\Gamma_n$-isometry) if and only if $(S_1, \dots, S_{n-1},P)$ is a $\Gamma_n$-contraction and $P$ is a unitary (isometry).
\end{thm}

It was shown in \cite{S:P} that every $\Gamma_n$-contraction $(S_1, \dots,S_{n-1},P)$ on $\mathcal{H}$ admits a canonical decomposition into a $\Gamma_n$-unitary $(S_{11}, \dots, S_{(n-1)1},P_1)$ and a $\Gamma_n$-contraction $(S_{12},\dots, S_{(n-1)2},P_2)$ for which $P_2$ is a c.n.u. (completely non-unitary) contraction. This reason naturally led to the following definition.
\begin{defn}\label{cnugamma}
	A commuting tuple of operators $(S_1, \dots, S_{n-1},P)$ on $\mathcal{H}$ is said to be a c.n.u. $\Gamma_n$-contraction if $(S_1, \dots, S_{n-1},P)$ is a $\Gamma_n$-contraction and $P$ is a c.n.u. contraction.
\end{defn}
There are several independent characterizations for the $\Gamma_n$-unitaries in the literature (see \cite{S:B, A:Pal, sourav14}). So, naturally we have keen interests in finding operator model for the c.n.u. $\Gamma_n$-contractions. In \cite{sourav14}, a Schaffer type explicit dilation and functional model were obtained for a $\Gamma_n$-contraction under certain conditions. In Theorem \ref{commutativemodel}, we construct an explicit Sz.-Nagy-Foias type $\Gamma_n$-isometric dilation for a certain c.n.u. $\Gamma_n$-contraction $(S_1, \dots , S_{n-1},P)$ on the minimal Sz.-Nagy-Foias isometric dilation space of $P$ and as a consequence we have a functional model for it. By an application of this model, we represent in Theorem \ref{representation} a c.n.u. $\Gamma_n$-contraction $(S_1, \dots , S_{n-1},P)$ as $S_i=C_i+PC_{n-i}^*$. The $\ft$-tuple is the main ingredient in these constructions. In the same Section, we find two different necessary conditions for this $\Gamma_n$-isometric dilation and as a consequence of this dilation we obtain a sufficient condition under which a commuting tuple $(S_1, \dots , S_{n-1},P)$ becomes a $\Gamma_n$-contraction.

Also, with the help of the $\ft$-tuples of $(S_1, \dots,S_{n-1},P)$ and $(S_1^*, \dots, S_{n-1}^*, P^*)$, we construct an explicit abstract operator model in Theorem \ref{noncommutemodel} for a c.n.u. $\Gamma_n$-contraction satisfying $S_i^*P=PS_i^*$ for each $i$. This model is not necessarily a commutative one and is different from the model of Theorem \ref{commutativemodel}. We give a counter example to show that the conclusion of Theorem \ref{noncommutemodel} is not true if we drop the condition $S_i^*P=PS_i^*$ for each $i$.

The characteristic function of Sz.-Nagy and Foias, \cite{Nagy} is a complete unitary invariant for a c.n.u. contraction. In Section \ref{completunitaryinvariant}, we show that the $\ft$-tuple and the characteristic function of $P$ constitute a complete unitary invariant for a c.n.u. $\Gamma_n$-contraction satisfying $S_i^*P=PS_i^*$. In Section \ref{lastSec}, we present an alternative proof to the famous Beurling-Lax-Halmos representation theorem (see \cite{S:B}) describing the joint-invariant subspaces of a pure $\Gamma_n$-isometry. Section \ref{commodel} deals with a brief literature of $\Gamma_n$-contraction and some necessary preparatory results.

\vspace{0.4cm}

\section{A brief literature and preparatory results}\label{commodel}

\vspace{0.4cm}

\noindent  As we have already mentioned that $\Gamma_n$-unitaries, $\Gamma_n$-isometries and $\Gamma_n$-co-isometries are special classes of $\Gamma_n$-contractions. We first recall their definitions from the literature (see \cite{S:B}).
\begin{defn}\label{def3}
	Let $(S_1, \dots, S_{n-1}, P)$ be a commuting tuple acting on a Hilbert space $\mathcal{H}$. We say that $(S_1, \dots, S_{n-1}, P)$ is
	\begin{itemize}
		\item [(i)] a $\Gamma_n$-\textit{unitary} if $S_1, \dots, S_{n-1} \text{ and }P$ are normal
		operators and the Taylor joint spectrum $\sigma(S_1, \dots, S_{n-1},P)$ is contained in
		the distinguished boundary $b\Gamma_n$ of $\Gamma_n$, where $b\Gamma_n = \{(s_1, \dots, s_{n-1}, p)\in \Gamma_n :
		|p|=1 \}$;
		
		\item [(ii)] a $\Gamma_n$-\textit{isometry} if there exist a
		Hilbert space $\mathcal K$ containing $\mathcal H$ as a closed subspace and a $\Gamma_n$-unitary $(\tilde{S}_1,\dots,\tilde{S}_{n-1},\tilde{P})$ on $\mathcal K$ such
		that $\mathcal H$ is a common invariant subspace of $\tilde{S}_1, \dots, \tilde{S}_{n-1}, \tilde{P}$ and
		that $S_i=\tilde{S}_i|_{\mathcal H}$ for each $i$ and
$		P=\tilde{P}|_{\mathcal H}$ ; 
		
		\item [(iii)] a $\Gamma_n$-\textit{co-isometry} if $(S_1^*,\dots, S_{n-1}^*,P^*)$ is a $\Gamma_n$-isometry.
	\end{itemize}
\end{defn}
Note that $D_P$ and $\mathcal D_P$ are the defect operator and defect space respectively of a contraction $P$ acting on a Hilbert space $\mathcal{H}$. Suppose $\Lambda_P = \{ z\in \mathbb C : (I - zP^*) \text{ is invertible}\}$. For $z \in \Lambda_P$, the \textit{characteristic function} of $P$ (see \cite{Nagy}) is defined as
\begin{equation}\label{char}
\Theta_P(z) = [-P + zD_{P^*}(I - zP^*)^{-1}D_P]|_{\mathcal{D}_P}\,.
\end{equation} 
By virtue of the relation $PD_P = D_{P^*}P$, $\Theta_P(z)$ maps $\mathcal{D}_P$ into $\mathcal{D}_{P^*}$ for every $z \in \Lambda_P$. Since $\mathbb{D} \subseteq \Lambda_P$, so for every $z \in \mathbb{D}$, $\Theta_P(z)$ maps $\mathcal{D}_P$ into $\mathcal{D}_{P^*}$\,.
Consider 
\begin{equation}\label{delta}
\Delta_P(t)=(I_{\mathcal{D}_P}-\Theta_P(e^{it})^*\Theta_P(e^{it}))^{1/2},
\end{equation}
 for those $t$ at which $\Theta_P(e^{it})$ exists, on $L^2(\mathcal{D}_P)$ and 
\[
\textbf{H}=H^2(\mathcal{D}_{P^*})\oplus \overline{\Delta_P(L^2(\mathcal{D}_P))}\ominus \{\Theta_Pu \oplus \Delta_Pu: u \in H^2(\mathcal{D}_P)\}.
\]

The canonical decomposition of a contraction separates the unitary part from its c.n.u. part. Sz.-Nagy and Foias obtained the following model for a c.n.u. contraction.

\begin{thm}[\textbf{Sz.-Nagy and Foias}, \cite{Nagy}]\label{nagymodel}
	Let $P$ be a c.n.u. contraction defined on a Hilbert space $\mathcal{H}$. Then $\mathcal{H}$ can be identified with $\textbf{H}$ and $P$ can be identified with $P_{\textbf{H}}(M_z \oplus M_{e^{it}}|_{\overline{\Delta_P(L^2(\mathcal{D}_P))}})|_{\textbf{H}}$.
\end{thm}

The following theorem due to Levan provides an orthogonal decomposition to a c.n.u. contraction in terms of a shift operator and a c.n.i. contraction.

\begin{thm}[{\cite{N:L}, Theorem 1}]\label{thmlv1}
	With respect to a c.n.u. contraction $T$ on $\mathcal{H}$, $\mathcal{H}$ admits the unique orthogonal decomposition
	$$\mathcal{H} = \mathcal{H}_1\oplus\mathcal{H}_2,$$
	where $T|_{\mathcal{H}_1}$ is c.n.u. isometry (i.e. a shift operator), while $T|_{\mathcal{H}_2}$ is c.n.i. contraction.
\end{thm}

Let $P$ be a c.n.u. contraction on $\mathcal{H}$. Then by Theorem \ref{thmlv1}, $\mathcal{H}$ can be decomposed as $\mathcal{H} = \mathcal{H}_1 \oplus\mathcal{H}_2$ and 
$
P
=
\begin{bmatrix}
P_1 & 0\\
0 & P_2
\end{bmatrix}
$	
on $\mathcal{H}_1 \oplus\mathcal{H}_2$, with $P_1$ on $\mathcal{H}_1$ being c.n.u. isometry and $P_2$ on $\mathcal{H}_2$ is c.n.i.. By Theorem \ref{nagymodel}, $\mathcal{H}$ can be identified with $\textbf{H}$.
Since $P_1$ on $\mathcal{H}_1$ is a pure isometry, $P_1$ can be identified with $M_z$ on $H^2(\mathcal{D}_{P_1^*})$. Again since $\mathcal{H}_1$ is identified with $H^2(\mathcal{D}_{P_1^*})$, $H^2(\mathcal{D}_{P_1^*})$ can be embeded in $\textbf{H}$ by an isometry, say $V_{\Delta} : H^2(\mathcal{D}_{P_1^*}) \hookrightarrow \textbf{H} $. Set $\mathcal{N} = \textbf{H} \ominus V_{\Delta}(H^2(\mathcal{D}_{P_1^*}))$. Therefore, $\mathcal{H}$ is identified with $\textbf{H}$ and hence with $H^2(\mathcal{D}_{P_1^*})\oplus\mathcal{N}$. Hence $P_1$ on $\mathcal{H}_1$ is identified with $M_z$ on $H^2(\mathcal{D}_{P_1^*})$ and $P_2$ on $\mathcal{H}_2$ is identified with $P_{\mathcal{N}}(M_{e^{it}})|_{\mathcal{N}}$ on $\mathcal{N}$. This is because by Nagy-Foias model for c.n.u. $P$, $P$ can be identified with $ P_{\textbf{H}}(M_z \oplus M_{e^{it}})|_{\textbf{H}}$, where $M_z$ is defined on $H^2(\mathcal{D}_{P^*})$. Thus, we have the following model theorem for a c.n.u. contraction.
\begin{thm}\label{thm19}
	Let $P$ be a c.n.u. contraction on $\mathcal{H}$ and let 
	$
	P
	=
	\begin{bmatrix}
	P_1 & 0\\
	0 & P_2
	\end{bmatrix}
	$ 
	be the Levan's decomposition of $P$ with respect to $\mathcal{H} = \mathcal{H}_1 \oplus \mathcal{H}_2$. Then $\mathcal{H}_1$ and $\mathcal{H}_2$ can be identified with $H^2(\mathcal{D}_{P_1^*})$ and $\mathcal{N}$ respectively and $P|_{\mathcal{H}_1}$ i.e., $P_1$ is identified with $M_z$ on $H^2(\mathcal{D}_{P_1^*})$ and $P|_{\mathcal{H}_2}$, i.e., $P_2$ is identified with $P_{\mathcal{N}}(M_{e^{it}})|_{\mathcal{N}}$ on $\mathcal{N}$ i.e., with $P_{\mathcal{N}}M_z|_{\mathcal{N}}$, where $M_z$ is the multiplication operator on $L^2(\mathcal{D}_P)$.
\end{thm}

The following theorem provides an Wold-type decomposition of a $\Gamma_n$-isometry.
\begin{thm}[\cite{S:B}, Theorem 4.12]\label{wold}
	Let $(S_1, \dots, S_{n-1},P)$ be a commuting tuple of $n$-operators on a Hilbert space $\mathcal{H}$. Then the following are equivalent.
	\begin{enumerate}
		\item $(S_1, \dots, S_{n-1},P)$ is a $\Gamma_n$-isometry;
		\item there exists an orthogonal decomposition $\mathcal{H}=\mathcal{H}_1 \oplus \mathcal{H}_2$ into common reducing subspaces of $S_1, \dots, S_{n-1},P$ such that $(S_1|_{\mathcal{H}_1}, \dots, S_{n-1}|_{\mathcal{H}_1},P|_{\mathcal{H}_1})$ is a $\Gamma_n$-unitary and $(S_1|_{\mathcal{H}_2},\\ \dots, S_{n-1}|_{\mathcal{H}_2},P|_{\mathcal{H}_2})$ is a pure $\Gamma_n$-isometry.
	\end{enumerate} 
\end{thm}
The next result gives an explicit model for pure $\Gamma_n$-isometry.
\begin{thm}[\cite{S:B}, Theorem 4.10 \& \cite{S:P1}, Theorem 5.6]\label{puremodel}
	Let $(\widetilde{S}_1, \dots, \widetilde{S}_{n-1}, \widetilde{P})$ be a commuting tuple of operators on a Hilbert space $\mathcal{H}$. If $(\widetilde{S}_1, \dots, \widetilde{S}_{n-1}, \widetilde{P})$ is a pure $\Gamma_n$-isometry then there is a unitary operator $U: \mathcal{H} \to H^2(\mathcal{D}_{\widetilde{P}^*})$ such that 
	\[
	\widetilde{S}_i = U^*T_{\varphi_i}U, \quad i = 1, \dots, n-1 \quad \text{and} \quad \widetilde{P}=U^*T_zU.
	\]
	Here each $T_{\varphi_i}$ is the Toepllitz operator on the vectorial Hardy space $H^2(\mathcal{D}_{\widetilde{P}^*})$ with the symbol $\varphi_i(z)= F_i^* + F_{n-i}z$, where $(F_1, \dots, F_{n-1})$ is the $\ft$-tuple of $(\widetilde{S}_1^*, \dots, \widetilde{S}_{n-1}^*,\widetilde{P}^*)$ such that 
	\[
	\left( \dfrac{n-1}{n}(F_1^* + F_{n-1}z), \dfrac{n-2}{n}(F_2^* + F_{n-2}z), \dots, \dfrac{1}{n}(F_{n-1}^*+F_1z) \right)
	\]
	is a $\Gamma_{n-1}$-contraction for every $z \in \overline{\mathbb D}$.\\
	
	Conversely, if $F_1, \dots, F_{n-1}$ are $n-1$ bounded operators on a Hilbert space $E$ such that 
	\[
	\left( \dfrac{n-1}{n}(F_1^* + F_{n-1}z), \dfrac{n-2}{n}(F_2^* + F_{n-2}z), \dots, \dfrac{1}{n}(F_{n-1}^* + F_1z) \right)
	\]
	is a $\Gamma_{n-1}$-contraction for every $z \in \overline{\mathbb D}$, then $(T_{F_1^* + F_{n-1}z}, \dots, T_{F_{n-1}^*+F_1z}, T_z)$ on $H^2(E)$ is a pure $\Gamma_n$-isometry.
\end{thm}
The following lemma is well known and we skip the proof as it is a routine exercise.
\begin{lem}\label{inter}
	Let $U, V$ be a unitary and a pure isometry on Hilbert space $\mathcal{H}_1$ and $\mathcal{H}_2$ respectively, and let $X:\mathcal{H}_1 \to \mathcal{H}_2$ be such that $XU=VX$. Then $X=0$.
\end{lem}

\vspace{0.4cm}

\section{Sz.-Nagy-Foias type dilation and a functional model for a class of c.n.u. $\Gamma_n$-contraction}

\vspace{0.4cm}

\noindent In this Section, we construct an explicit $\Gamma_n$-isometric dilation for a certain class of c.n.u. $\Gamma_n$-contractions $(S_1,\dots , S_{n-1},P)$ on the Sz.-Nagy-Foias minimal isometric dilation space of the c.n.u. contraction $P$. Naturally, such a dilation becomes minimal. As a consequence of this dilation, we obtain a functional model for such c.n.u. $\Gamma_n$-contraction. The $\ft$-tuples are the main ingredients in the construction of such dilation and model. Also, as an application of this model we are able to express each $S_i$ as $ S_i = C_i + PC_{n-i}^* $ for some operator $C_i$. We begin with an elementary lemma which will be used in the proof of the model theorem.

\begin{lem}\label{unitary}
	Let $E_i, K_i$ for $i=1,2$, be Hilbert spaces. Suppose $U$ is a unitary from $(H^2(\mathbb D)\otimes E_1) \oplus K_1$ to $(H^2(\mathbb D)\otimes E_2) \oplus K_2$ such that 
	\[
	(M_z \otimes I_{E_1}) \oplus W_1 = U((M_z \otimes I_{E_2}) \oplus W_2)U^*,
	\]
	where $W_1$ on $K_1$ and $W_2$ on $K_2$ are unitaries. Then $U$ is of the form $(I_{H^2(\mathbb D)}\otimes U_1)\oplus U_2$ for some unitaries $U_1:E_1 \to E_2$ and $U_2:K_1 \to K_2$.
\end{lem}
\begin{proof}
	Suppose the block matrix form of $U$ is $\begin{pmatrix}
	U_{11} & U_{12}\\
	U_{21} & U_{22}
	\end{pmatrix}$. Since 
	\begin{equation}\label{eqn}
	\left((M_z \otimes I_{E_1})\oplus W_1\right)U = U \left((M_z \otimes I_{E_2}) \oplus W_2\right),
	\end{equation}
	that is,
	\[
	\begin{pmatrix}
	M_z \otimes I_{E_1} & 0\\
	0 & W_1
	\end{pmatrix}\begin{pmatrix}
	U_{11} & U_{12}\\
	U_{21} & U_{22}
	\end{pmatrix}=\begin{pmatrix}
	U_{11} & U_{12}\\
	U_{21} & U_{22}
	\end{pmatrix}\begin{pmatrix}
	M_z \otimes I_{E_2} & 0 \\
	0 & W_2
	\end{pmatrix},
	\]
	so 
	\begingroup
	\allowdisplaybreaks
	\begin{align*}
		(M_z \otimes I_{E_1})U_{11} &= U_{11}(M_z \otimes I_{E_2}) \quad &(M_z \otimes I_{E_1})U_{12}&=U_{12}W_2\\
		W_1U_{21}&=U_{21}(M_z \otimes I_{E_2}) \quad &W_1U_{22} &=U_{22}W_2.
	\end{align*}
	\endgroup
	Since $(M_z\otimes I_{E_1})$ is a pure isometry and $W_2$ is a unitary, by Lemma \ref{inter}, $U_{12}=0$. Thus 
	\[
	U=
	\begin{pmatrix}
		U_{11} & 0\\
		U_{21}& U_{22}
	\end{pmatrix}.
	\]
	Again $U^*$ is a unitary which satisfies the following intertwining relation
	\[
	\begin{pmatrix}
	M_z\otimes I_{E_1} & 0\\
	0 & W_1
	\end{pmatrix}\begin{pmatrix}
	U_{11}^* & U_{21}^*\\
	0 & U_{22}^*
	\end{pmatrix}= \begin{pmatrix}
	U_{11}^* & U_{21}^*\\
	0 & U_{22}^*
	\end{pmatrix}\begin{pmatrix}
	M_z\otimes I_{E_2} & 0\\
	0 & W_2
	\end{pmatrix}.
	\]
	This implies that $(M_z\otimes I_{E_1})U_{21}^*=U_{21}^*W_2$, where $M_z\otimes I_{E_1}$ is a pure isometry and $W_2$ is a unitary. Again applying Lemma \ref{inter}, we have that $U_{21}=0$. Hence $U=\begin{pmatrix}
	U_{11} & 0\\
	0 & U_{22}
	\end{pmatrix}.$ From \eqref{eqn} we get $(M_z\otimes I_{E_1})U_{11}=U_{11}(M_z\otimes I_{E_2})$, where $U_{11}$ is a unitary from $H^2(\mathbb D)\otimes E_1$ to $H^2(\mathbb D)\otimes E_2$. It is noticeable that for any unitary $U_1: E_1\to E_2$, $I_{H^2(\mathbb D)}\otimes U_1$ is a unitary from $H^2(\mathbb D)\otimes E_1$ to $H^2(\mathbb D)\otimes E_2$ that intertwines $M_z\otimes I_{E_1}$ and $M_z\otimes I_{E_2}$. Therefore, without loss of generality we can choose $U_{11}= I_{H^2(\mathbb D)}\otimes U_1$. Set $U_2=U_{22}$. Thus $U$ is of the form $(I_{H^2(\mathbb D)}\otimes U_1)\oplus U_2$.
	
\end{proof}

For a contraction $P$ on $\mathcal{H}$, consider the following Hilbert spaces: 
\begingroup
\allowdisplaybreaks
\begin{align}
	&\textbf{K}_+=H^2(\mathcal{D}_{P^*})\oplus \overline{\Delta_P(L^2(\mathcal{D}_P))};\label{eqK}\\
	&\textbf{H}=\textbf{K}_+\ominus \{\Theta_Pu \oplus \Delta_Pu: u \in H^2(\mathcal{D}_P)\}, \label{eqH}
\end{align}
\endgroup
where $\Theta_P,\Delta_P$ are as in \eqref{char} and \eqref{delta} respectively. The following theorem provides a minimal isometric dilation and model for a c.n.u. contraction.

\begin{thm}[\textbf{Sz.-Nagy and Foias}, \cite{Nagy}]
	Let $P$ be a c.n.u. contraction on $\mathcal{H}$. Then $(M_z \oplus M_{e^{it}}|_{\overline{\Delta_P(L^2(\mathcal{D}_P))}})(=\textbf{V})$ is the minimal isometric dilation of $P$ on $\textbf{K}_+$. The space $\mathcal{H}$ can be identified with $\textbf{H}$ and $P$ is unitarily equivalent to $P_{\textbf{H}}(M_z\oplus M_{e^{it}}|_{\overline{\Delta_P(L^2(\mathcal{D}_P))}})|_{\textbf{H}}$.
\end{thm}
The space $\textbf{K}_+$ can be identified with $\widetilde{\textbf{K}_+}=(H^2(\mathbb D)\otimes \mathcal{D}_{P^*})\oplus \overline{\Delta_P(L^2(\mathcal{D}_P))}$ by the natural unitary $U \oplus I_{\overline{\Delta_P(L^2(\mathcal{D}_P))}}$, where the unitary $U: H^2(\mathcal{D}_{P^*})\to H^2(\mathbb D)\otimes \mathcal{D}_{P^*}$ defined by 
\[
U(z^nh) = z^n \otimes h, \qquad (n \in \mathbb N \cup \{0\}, h \in \mathcal{D}_{P^*}).
\]
Set $\widetilde{\textbf{H}}= (U \oplus I_{\overline{\Delta_P(L^2(\mathcal{D}_P))}})\textbf{H}$ and $\widetilde{\textbf{V}}=(U \oplus I_{\overline{\Delta_P(L^2(\mathcal{D}_P))}})\textbf{V}(U \oplus I_{\overline{\Delta_P(L^2(\mathcal{D}_P))}})^{-1}$.
Then $\widetilde{\textbf{V}}$ on $\widetilde{\textbf{K}_+}$ is the minimal isometric dilation of $P$.
Before going to present the main result of this section let us recall Schaffer-type $\Gamma_n$-isometric dilation of a $\Gamma_n$-contraction from \cite{sourav14}. This is a variant of Theorem 6.3 in \cite{sourav14}.
\begin{thm} \label{spdilation}
	Let $(S_1, \dots, S_{n-1},P)$ be a $\Gamma_n$-contraction defined on a Hilbert space $\mathcal{H}$ with $(A_1, \dots, A_{n-1})$ and $(B_1, \dots,B_{n-1})$ being the $\ft$-tuples of $(S_1, \dots, S_{n-1},P)$ and $(S_1^*, \dots, S_{n-1}^*,P^*)$ respectively. Let $\mathcal{K}= \mathcal{H} \oplus \mathcal{D}_P \oplus \mathcal{D}_P \oplus \dots $ and let $(R_1, \dots, R_{n-1},V)$ be defined on $\mathcal{K}$ by 
	\[
	R_i = \begin{pmatrix}
	S_i & 0 & 0 & 0 & \cdots\\
	A_{n-i}^*D_P & A_i & 0 & 0 & \cdots\\
	0 & A_{n-i}^* & A_i & 0 & \cdots\\
	0 & 0 & A_{n-i}^* & A_i & \cdots\\
	\vdots & \vdots & \vdots & \vdots & \ddots
	\end{pmatrix}, V = \begin{pmatrix}
	P & 0 & 0 & 0 & \cdots\\
	D_P & 0 & 0 & 0 & \cdots\\
	0 & I & 0 & 0 & \cdots\\
	0 & 0 & I & 0 & \cdots\\
	\vdots & \vdots & \vdots & \vdots & \ddots
	\end{pmatrix}.
	\]
	Suppose 
	\[
	\Sigma_1(z)=\left( \dfrac{n-1}{n}(A_1 + A_{n-1}^*z), \dfrac{n-2}{n}(A_{2} + A_{n-2}^*z), \dots, \dfrac{1}{n}(A_{n-1} + A_1^*z) \right)
	\]
	and 
	\[
	\Sigma_2(z)=\left( \dfrac{n-1}{n}(B_1^* + B_{n-1}z), \dfrac{n-2}{n}(B_{2}^* + B_{n-2}z), \dots, \dfrac{1}{n}(B_{n-1}^* + B_1z) \right)
	\]
	If  \,$\Sigma_1(z)$ and $\Sigma_2(z)$ are $\Gamma_{n-1}$-contractions for every $z\in \T$, then $(R_1, \dots, R_{n-1},V)$ is a minimal $\Gamma_n$-isometric dilation of $(S_1, \dots, S_{n-1},P)$. Moreover, $(R_1^*, \dots , R_{n-1}^*,V^*)$ is a $\Gamma_n$-co-isometric extension of $(S_1^*, \dots , S_{n-1}^*,P^*)$. 
\end{thm}

As a consequence of the previous dilation theorem we obtain the following sufficient condition under which a commuting tuple $(S_1, \dots, S_{n-1},P)$ to become a $\Gamma_n$-contraction.
\begin{thm}\label{gammasufficient}
	Let $\Sigma=(S_1, \dots, S_{n-1},P)$ be a commuting tuple of operators on $\mathcal{H}$ such that there are unique operators $A_1, \dots , A_{n-1}$ in $\mathcal B(\mathcal D_P)$ satisfying $S_i-S_{n-i}^*P = D_PA_iD_P$, for $1\leq i \leq n-1$. Then $\Sigma$ is a $\Gamma_n$-contraction if 
	\begin{enumerate}
		\item $$\Big(\dfrac{n-1}{n}(A_1+A_{n-1}^*z), \dfrac{n-2}{n}(A_2+A_{n-2}^*z), \dots, \dfrac{1}{n}(A_{n-1}+A_1^*z)\Big)$$ and $$\Big(\dfrac{n-1}{n}(B_1^*+B_{n-1}z),\dfrac{n-2}{n}(B_2^*+B_{n-2}z, \dots, \dfrac{1}{n}(B_{n-1}^*+B_1z))\Big)$$ are $\Gamma_{n-1}$-contractions for all $z \in \T$;
		\item $\left(\dfrac{n-1}{n}S_1, \dots, \dfrac{1}{n}S_{n-1}\right)$ is a $\Gamma_{n-1}$-contraction.
	\end{enumerate}
\end{thm}
\begin{proof}
	Let $R_1, \dots,R_{n-1},V$ be as in Theorem \ref{spdilation}. If $(S_1, \dots, S_{n-1},P)$ satisfies the given hypotheses, then by Theorem \ref{spdilation}, $(R_1, \dots,R_{n-1},V)$ on $\mathcal{K}$ is a $\Gamma_n$-isometry such that $(R_1^*, \dots,R_{n-1}^*,V^*)$ extends $(S_1^*, \dots, S_{n-1}^*,P^*)$. Then $(S_1^*,\dots, S_{n-1}^*,P^*)$ is a $\Gamma_n$-contraction by being the restriction of a $\Gamma_n$-co-isometry to a joint-invariant subspace. And hence $(S_1, \dots, S_{n-1},P)$ is a $\Gamma_n$-contraction.
\end{proof}

Consider $(S_1, \dots, S_{n-1},P)$ is a c.n.u. $\Gamma_n$-contraction on $\mathcal{H}$ which satisfies the hypothesis of Theorem \ref{spdilation}, that is, $\Sigma_1(z)$ and $\Sigma_2(z)$ are $\Gamma_{n-1}$-contractions for every $z\in \mathbb D$. Then by Theorem \ref{spdilation}, $(R_1, \dots, R_{n-1},V)$ on $\mathcal{K}$ is the minimal $\Gamma_n$-isometric dilation of $(S_1, \dots, S_{n-1},P)$. Again by Theorem \ref{wold}, there exists an orthogonal decomposition $\mathcal{K}=\mathcal{K}_1 \oplus \mathcal{K}_2$ such that $\mathcal{K}_1$ and $\mathcal{K}_2$ reduce each $R_i$ and $V$ and $(R_{11}, \dots, R_{1(n-1)}, V_1)=(R_1|_{\mathcal{K}_1}, \dots, R_{n-1}|_{\mathcal{K}_1}, V|_{\mathcal{K}_1})$ is a pure $\Gamma_n$-isometry on $\mathcal{K}_1$ while $(R_{21}, \dots, R_{2(n-1)}, V_2)=(R_1|_{\mathcal{K}_2}, \dots, R_{n-1}|_{\mathcal{K}_2}, V|_{\mathcal{K}_2})$ is a $\Gamma_n$-unitary. Using Theorem \ref{spdilation} and the above fact we have the following dilation and commutative model for a c.n.u. $\Gamma_n$-contraction $(S_1,\dots, S_{n-1},P)$ with $\Sigma_1(z)$ and $\Sigma_2(z)$ are $\Gamma_{n-1}$-contractions for every $z\in \mathbb D$.

\begin{thm}\label{commutativemodel}
	Let $(S_1, \dots, S_{n-1},P)$ be a c.n.u. $\Gamma_n$-contraction on a Hilbert space $\mathcal{H}$ with $(A_1, \dots, A_{n-1})$ and $(B_1, \dots, B_{n-1})$ being the $\ft$-tuples of $(S_1, \dots, S_{n-1},P)$ and $(S_1^*, \dots, S_{n-1}^*,P^*)$ respectively. Let $\Sigma_1(z)$ and $\Sigma_2(z)$ be $\Gamma_{n-1}$-contractions for all $z \in \mathbb D$. Suppose $(G_1, \dots, G_{n-1})$ is the $\ft$-tuple of $(R_{11}^*, \dots, R_{1(n-1)}^*,V_1^*)$. Then $\big(\widetilde{S}_{11}\oplus \widetilde{S}_{12}, \dots, \widetilde{S}_{(n-1)1}\oplus \widetilde{S}_{(n-1)2}, \widetilde{P}_1\oplus \widetilde{P}_2\big)$ on $\widetilde{\textbf{K}_+}$ is the minimal $\Gamma_n$-isometric dilation of $(S_1, \dots, S_{n-1},P)$, where 
	\[
	\widetilde{S}_{i1}=I_{H^2(\mathbb D)}\otimes U_1^*G_i^*U_1 + M_z \otimes U_1^*G_{n-i}U_1,\, \widetilde{P}_1= M_z \otimes I_{\mathcal{D}_{P^*}} \;\, \text{on }H^2(\mathbb D)\otimes \mathcal{D}_{P^*}
	\]
	and 
	\[
	\widetilde{S}_{i2}=U_2^*R_{2i}U_2,\;\; \widetilde{P}_2=M_{e^{it}}|_{\overline{\Delta_P(L^2(\mathcal{D}_P))}} \quad \text{on }\overline{\Delta_P(L^2(\mathcal{D}_P))},
	\]
	for some unitaries $U_1: \mathcal{D}_{P^*}\to \mathcal{D}_{V^*}$ and $U_2: \overline{\Delta_P(L^2(\mathcal{D}_P))} \to \mathcal{K}_2$. Moreover, for $i= 1, \dots, n-1$
	\[
	S_i = P_{\mathcal{H}}\big(\widetilde{S}_{i1}\oplus \widetilde{S}_{i2}\big)|_{\mathcal{H}} \text{  and }P = P_{\mathcal{H}}\big((M_z\otimes I_{\mathcal{D}_{P^*}})\oplus M_{e^{it}}|_{\overline{\Delta_P(L^2(\mathcal{D}_P))}}\big)|_{\mathcal{H}}.
	\]
\end{thm}
\begin{proof}
	Let $(R_1, \dots, R_{n-1},V)$ on $\mathcal{K}$ be a minimal $\Gamma_n$-isometric dilation of the $\Gamma_n$-contraction $(S_1, \dots, S_{n-1},P)$ as in Theorem \ref{spdilation}. As stated before of this theorem, $(R_{11}, \dots, R_{1(n-1)}, V_1)$ on $\mathcal{K}_1$ is a pure $\Gamma_n$-isometry while $(R_{21}, \dots, R_{2(n-1)}, V_2)$ on $\mathcal{K}_2$ is a $\Gamma_n$-unitary on $\mathcal{K}_2$. Then by Theorem \ref{puremodel}, $(R_{11}, \dots, R_{1(n-1)}, V_1)$ on $\mathcal{K}_1$ and $\big(I_{H^2(\mathcal{D})}\otimes G_1^* + M_z \otimes G_{n-1},\dots, I_{H^2(\mathbb D)}\otimes G_{n-1}^* + M_z \otimes G_1, M_z \otimes I_{\mathcal{D}_{V_1^*}}\big)$ on $H^2(\mathbb D)\otimes \mathcal{D}_{V_1^*}$ $(=H ^2(\mathbb D)\otimes \mathcal{D}_{V^*})$ are unitarily equivalent by a unitary say $\varphi: \mathcal{K}_1 \to H^2(\mathbb D)\otimes \mathcal{D}_{V^*}$. Let 
	\[
	{K}'= (H^2(\mathbb D)\otimes \mathcal{D}_{V^*}) \oplus \mathcal{K}_2.
	\]
	Then $\mathcal{K}$ and $K'$ are unitarily equivalent by the unitary $\varphi \oplus I_{\mathcal{K}_2}$. Hence $\big((I_{H^2(\mathcal{D})}\otimes G_1^* + M_z \otimes G_{n-1})\oplus R_{21},\dots, ( I_{H^2(\mathbb D)}\otimes G_{n-1}^* + M_z \otimes G_1)\oplus R_{2(n-1)},  (M_z \otimes I_{\mathcal{D}_{V_1^*}})\oplus V_2\big)$ on $K'$ is the minimal $\Gamma_n$-isometric dilation of $(S_1, \dots, S_{n-1},P)$. Therefore, $(M_z \otimes I_{\mathcal{D}_{V_1^*}})\oplus V_2$ on $K'$ is the minimal isometric dilation of $P$ and hence $(M_z \otimes I_{\mathcal{D}_{V_1^*}})\oplus V_2$ on $K'$ and $(M_z\otimes I_{\mathcal{D}_{P^*}})\oplus M_{e^{it}}|_{\overline{\Delta_P(L^2(\mathcal{D}_P))}}$ on $\widetilde{\textbf{K}_+}$ are unitarily equivalent. Let $U:\widetilde{\textbf{K}_+}\to K'$ be a unitary such that 
	\begin{equation}\label{eqU}
	(M_z \otimes I_{\mathcal{D}_{V_1^*}})\oplus V_2 = U[(M_z\otimes I_{\mathcal{D}_{P^*}})\oplus M_{e^{it}}|_{\overline{\Delta_P(L^2(\mathcal{D}_P))}}]U^*.
	\end{equation}
	Then by Lemma \ref{unitary}, $U$ is of the form $(I_{H^2(\mathbb D)}\otimes U_1) \oplus U_2$ for some unitaries $U_1: \mathcal{D}_{P^*}\to \mathcal{D}_{V^*}$ and $U_2:\overline{\Delta_P(L^2(\mathcal{D}_P))} \to \mathcal{K}_2$. Consider for each $i=1, \dots, n-1$
	\begingroup
	\allowdisplaybreaks
	\begin{align*}
	\widetilde{S}_{i1}&= (I_{H^2(\mathbb D)}\otimes U_1)^*(I_{H^2(\mathbb D)}\otimes G_i^*\oplus M_z \otimes G_{n-i})(I_{H^2(\mathbb D)}\otimes U_1)\\
	& \qquad\qquad\qquad\qquad\qquad\text{ and }\\
	\widetilde{P}_1&=(I_{H^2(\mathbb D)}\otimes U_1)^*(M_z \otimes I_{\mathcal{D}_{V_1^*}})(I_{H^2(\mathbb D)}\otimes U_1).
	\end{align*}
	\endgroup
	Again from Equation \eqref{eqU}, we have $M_{e^{it}}|_{\overline{\Delta_P(L^2(\mathcal{D}_P))}}=U_2^*V_2U_2$ and $\widetilde{P}_1=(I_{H^2(\mathbb D)}\otimes U_1)^*(M_z \otimes I_{\mathcal{D}_{V_1^*}})(I_{H^2(\mathbb D)}\otimes U_1) = M_z \otimes I_{\mathcal{D}_{P^*}}$. Since $\big((I_{H^2(\mathcal{D})}\otimes G_1^* + M_z \otimes G_{n-1}),\dots, (I_{H^2(\mathbb D)}\otimes G_{n-1}^* + M_z \otimes G_{1}),  M_z \otimes I_{\mathcal{D}_{V_1^*}}\big)$ on $H ^2(\mathbb D)\otimes \mathcal{D}_{V^*}$ is a pure $\Gamma_n$-isometry and $(\widetilde{S}_{11}, \dots, \widetilde{S}_{(n-1)1}, \widetilde{P}_1)$ on $H^2(\mathbb D)\otimes \mathcal{D}_{P^*}$ is unitarily equivalent to $\big((I_{H^2(\mathcal{D})}\otimes G_1^* + M_z \otimes G_{n-1}),\dots, (I_{H^2(\mathbb D)}\otimes G_{n-1}^* + M_z \otimes G_{1}), M_z \otimes I_{\mathcal{D}_{V_1^*}}\big)$, so $(\widetilde{S}_{11}, \dots, \widetilde{S}_{(n-1)1}, \widetilde{P}_1)$ is a pure $\Gamma_n$-isometry. Again by the hypothesis we have that
	\[
	\widetilde{S}_{i2} = U_2^*R_{2i}U_2, \quad \widetilde{P}_2 = U_2^*V_2U_2=M_{e^{it}}|_{\overline{\Delta_P(L^2(\mathcal{D}_P))}}.
	\]
	Therefore,  $(\widetilde{S}_{12}, \dots, \widetilde{S}_{(n-1)2}, \widetilde{P}_2)$ on $\overline{\Delta_P(L^2(\mathcal{D}_P))}$ is unitarily equivalent to $(R_{21}, \dots, R_{(n-1)2}, V_2)$ on $\mathcal{K}_2$ and thus $(\widetilde{S}_{12}, \dots, \widetilde{S}_{(n-1)2}, \widetilde{P}_2)$ is a $\Gamma_n$-unitary. Thus $(\widetilde{S}_{11}\oplus \widetilde{S}_{12}, \dots, \widetilde{S}_{(n-1)1}\oplus \widetilde{S}_{(n-1)2}, \widetilde{P}_1\oplus \widetilde{P}_2)$ on $\widetilde{\textbf{K}_+}$ is unitarily equivalent to $(R_1, \dots, R_{n-1},V)$ on $\mathcal{K}$ and hence $(\widetilde{S}_{11}\oplus \widetilde{S}_{12}, \dots, \widetilde{S}_{(n-1)1}\oplus \widetilde{S}_{(n-1)2}, \widetilde{P}_1\oplus \widetilde{P}_2)$ is the minimal $\Gamma_n$-isometric dilation of $(S_1, \dots, S_{n-1},P)$ as $(R_1, \dots, R_{n-1},V)$ on $\mathcal{K}$ is the minimal $\Gamma_n$-isometric dilation of $(S_1, \dots, S_{n-1},P)$. Thus $\mathcal{H}$ can be considered as a subspace of $\widetilde{\textbf{K}_+}$ and we have 
	\[
	S_i = P_{\mathcal{H}}\big(\widetilde{S}_{i1}\oplus \widetilde{S}_{i2}\big)|_{\mathcal{H}} \text{  and  }P= P_{\mathcal{H}}\big((M_z \otimes I_{\mathcal{D}_{P^*}})\oplus M_{e^{it}}|_{\overline{\Delta_P(L^2(\mathcal{D}_P))}}\big)|_{\mathcal{H}}.
	\]
\end{proof}

A direct consequence of the previous theorem is the following.

\begin{cor}\label{bisaipaldilation}
	Let $(S_1, \dots, S_{n-1},P)$ be a c.n.u. $\Gamma_n$-contraction on a Hilbert space $\mathcal{H}$ with $(A_1, \dots, A_{n-1})$ and $(B_1, \dots, B_{n-1})$ being the $\mathcal{F}_O$-tuples of $(S_1, \dots, S_{n-1},P)$ and $(S_1^*, \dots, S_{n-1}^*,P^*)$ respectively. Let $\Sigma_1(z)$ and $\Sigma_2(z)$ be $\Gamma_{n-1}$-contractions for all $z\in \mathbb D$. Then there is an isometry $\varphi_{BS}$ from $\mathcal{H}$ into $H^2(\mathcal{D}_{P^*})\oplus\overline{\Delta_P L^2(\mathcal{D}_P)}$ such that for $i=1, \dots, n-1$
	\[
	\varphi_{BS}S_i^*=\left(M_{F_i^*+zF_{n-i}}\oplus\widetilde{S}_{i2}\right)^*\varphi_{BS} \text{ and }\varphi_{BS}P^*=\left(M_z\oplus M_{e^{it}}|_{\overline{\Delta_P L^2(\mathcal{D}_P)}}\right)^*\varphi_{BS},
	\]
	where $(M_{F_1^*+zF_{n-1}}, \dots, M_{F_{n-1}^*+zF_1},M_z)$ on $H^2(\mathcal{D}_{P^*})$ is a pure $\Gamma_n$-isometry and $(\widetilde{S}_{12},\dots, \widetilde{S}_{(n-1)2},\\M_{e^{it}}|_{\overline{\Delta_P L^2(\mathcal{D}_P)}})$ on $\overline{\Delta_P L^2(\mathcal{D}_P)}$ is a $\Gamma_n$-unitary.
\end{cor}
\begin{proof}
	By Theorem \ref{commutativemodel}, $\big(\widetilde{S}_{11}\oplus \widetilde{S}_{12}, \dots, \widetilde{S}_{(n-1)1}\oplus \widetilde{S}_{(n-1)2}, \widetilde{P}_1\oplus \widetilde{P}_2\big)$ on $\widetilde{\textbf{K}_+}$ is the minimal $\Gamma_n$-isometric dilation of $(S_1, \dots, S_{n-1},P)$. Then by definition of a $\Gamma_n$-isometric dilation, there exists an isometry $\varphi_{BS}:\mathcal{H} \to \widetilde{\textbf{K}_+}$ such that 
	\[
	\varphi_{BS}P^*=\left(\widetilde{P}_1\oplus \widetilde{P}_2 \right)^*\varphi_{BS} \text{ and }\varphi_{BS}S_i^*=\left(\widetilde{S}_{i1}\oplus \widetilde{S}_{i2}\right)^*\varphi_{BS}, \text{ for all }i=1, \dots, n-1.
	\]
	Again by Theorem \ref{commutativemodel}, $\widetilde{S}_{i1}=I_{H^2(\mathbb D)}\otimes U_1^*G_i^*U_1 + M_z \otimes U_1^*G_{n-i}U_1$ and $\widetilde{P}_1= M_z\otimes I_{\mathcal{D}_{P^*}}$. Suppose $F_i=U_1^*G_iU_1$, for $i=1, \dots, n-1$. Invoking Theorem \ref{commutativemodel} again, $$\left(\widetilde{S}_{11}, \dots, \widetilde{S}_{(n-1)1}, \widetilde{P}_1\right)=\left(M_{F_1^*+zF_{n-1}}, \dots, M_{F_{n-1}^*+zF_1},M_z\right)$$
	on $H^2(\mathbb{D})\otimes \mathcal{D}_{P^*}$ is a pure $\Gamma_n$-isometry and $\left(\widetilde{S}_{12}, \dots, \widetilde{S}_{(n-1)2}, M_{e^{it}}|_{\overline{\Delta_P L^2(\mathcal{D}_P)}}\right)$ on $\overline{\Delta_P L^2(\mathcal{D}_P)}$ is a $\Gamma_n$-unitary.
\end{proof}

As a corollary of Theorem \ref{commutativemodel}, we have the following representation for such c.n.u. $\Gamma_n$-contraction.

\begin{thm}\label{representation}
	Let $(S_1, \dots, S_{n-1},P)$ be a c.n.u. $\Gamma_n$-contraction on a Hilbert space $\mathcal{H}$ with $(A_1, \dots, A_{n-1})$ and $(B_1,\dots, B_{n-1})$ being the $\ft$-tuples of $(S_1, \dots, S_{n-1},P)$ and $(S_1^*, \dots, S_{n-1}^*,P^*)$ respectively. Let $\Sigma_1(z)$ and $\Sigma_2(z)$ be $\Gamma_{n-1}$-contractions for all $z \in \mathbb D$. Consider $C_i=P_{\mathcal{H}}\big((I_{H^2(\mathbb D)}\\\otimes U_1^*G_iU_1) \oplus \dfrac{\widetilde{S}_{i2}}{2}\big)|_{\mathcal{H}}$ for $i= 1, \dots, n-1$, where $U_1, G_i, \widetilde{S}_{i2}$ are in Theorem \ref{commutativemodel}. Then $S_i = C_i + PC_{n-i}^*$. 
\end{thm}

\begin{proof}
	We have $\big(\widetilde{S}_{11} \oplus \widetilde{S}_{12}, \dots, \widetilde{S}_{(n-1)1}\oplus \widetilde{S}_{(n-1)2}, \widetilde{P}_1 \oplus \widetilde{P}_2\big)$ on $\widetilde{\textbf{K}_+}$ is the minimal $\Gamma_n$-isometric dilation of $(S_1, \dots, S_{n-1},P)$, by Theorem  \ref{commutativemodel}. Again by Theorem \ref{spcoextn}, $\big((\widetilde{S}_{11} \oplus \widetilde{S}_{12})^*, \dots, (\widetilde{S}_{(n-1)1}\oplus \widetilde{S}_{(n-1)2})^*, (\widetilde{P}_1 \oplus \widetilde{P}_2)^*\big)$ is a $\Gamma_n$-co-isometric extension of $(S_1^*, \dots,S_{n-1}^*,P^*)$ and hence $\mathcal{H}$ is invariant under $(\widetilde{S}_{i1}\oplus \widetilde{S}_{i2})^*$ and $(\widetilde{P}_1\oplus\widetilde{P}_2)^*$. Since $(\widetilde{S}_{12}, \dots, \widetilde{S}_{(n-1)2},\widetilde{P}_2)$ is a $\Gamma_n$-unitary, so $\widetilde{S}_{i2}^*=\widetilde{S}_{(n-i)2}\widetilde{P}_2^*$. Then 
	\begingroup
	\allowdisplaybreaks
	\begin{align*}
	S_i^* &= ((I_{H^2(\mathbb D)}\otimes U_1^*G_i^*U_1 + M_z \otimes U_1^*G_{n-i}U_1)\oplus \widetilde{S}_{i2})^*|_{\mathcal{H}}\\
	&=\left( ((I_{H^2(\mathbb D)}\otimes U_1^*G_iU_1)\oplus \dfrac{\widetilde{S}_{i2}^*}{2} \right) + \left( (I_{H^2(\mathbb D)}\otimes U_1^*G_{n-i}^*U_1)\oplus \dfrac{\widetilde{S}_{(n-i)2}}{2} \right)\\
	& \qquad\qquad\qquad\qquad\qquad\qquad\qquad((M_z^*\otimes I_{\mathcal{D}_{P^*}})\oplus M_{e^{it}}^*|_{\overline{\Delta_P(L^2(\mathcal{D}_P))}}) \big)|_{\mathcal{H}}\\
	&= P_{\mathcal{H}} \big((I_{H^2(\mathbb D)}\otimes U_1^*G_iU_1)\oplus \dfrac{\widetilde{S}_{i2}^*}{2}\big)|_{\mathcal{H}} + P_{\mathcal{H}}\big((I_{H^2(\mathbb D)}\otimes U_1^*G_{n-i}^*U_1)\\
	& \qquad\qquad\qquad\qquad\quad \oplus\dfrac{\widetilde{S}_{(n-i)2}}{2}\big)
	\big((M_z^*\otimes I_{\mathcal{D}_{P^*}})\oplus M_{e^{it}}^*|_{\overline{\Delta_P(L^2(\mathcal{D}_P))}}\big)|_{\mathcal{H}}\\
	& = P_{\mathcal{H}}\big( (I_{H^2(\mathbb D)}\otimes U_1^*G_iU_1)\oplus \dfrac{\widetilde{S}_{i2}^*}{2}\big)|_{\mathcal{H}} + P_{\mathcal{H}}\big((I_{H^2(\mathbb D)}\otimes U_1^*G_{n-i}^*U_1)\\
	&\qquad\qquad\quad\quad\oplus \dfrac{\widetilde{S}_{(n-i)2}}{2}\big)|_{\mathcal{H}}
	 P_{\mathcal{H}}\big((M_z^*\otimes I_{\mathcal{D}_{P^*}})\oplus M_{e^{it}}^*|_{\overline{\Delta_P(L^2(\mathcal{D}_P))}}\big) \big)|_{\mathcal{H}}.
	\end{align*}
	\endgroup
	The last equality follows from the fact that $\mathcal{H}$ is an invariant subspace for $(M_z^* \otimes I_{\mathcal{D}_{P^*}}\oplus M_{e^{it}}^*|_{\overline{\Delta_P(L^2(\mathcal{D}_P))}})$. Thus $S_i^* = C_i^* + C_{n-i}P^*$ and hence $S_i= C_i + PC_{n-i}^*$.
\end{proof}

\noindent \textbf{Note.} The operators $C_i$ for $i=1, \dots, n-1$, as in the above theorem, may not be unique. For example, suppose $n$ is odd. If we replace $\dfrac{\widetilde{S}_{i2}}{2}$ by $\alpha \widetilde{S}_{i2}$ in $C_i$ and $\dfrac{\widetilde{S}_{(n-i)2}}{2}$ by $(1-\alpha)\widetilde{S}_{(n-i)2}$ in $C_{n-i}$ for $i= 1, \dots, \dfrac{n-1}{2}$ and for any $\alpha \in (0,1)$, then also $S_i= C_i + PC_{n-i}^*$. Suppose $n$ is even. If we take $\alpha \widetilde{S}_{i2}$ instead of $\dfrac{\widetilde{S}_{i2}}{2}$ in $C_i$ and $(1-\alpha)\widetilde{S}_{(n-i)2}$ instead of $\dfrac{\widetilde{S}_{(n-i)2}}{2}$ in $C_{n-i} $ for $i=1, \dots, \dfrac{n}{2}-1$, then also $S_i = C_i+PC_{n-i}^*$.\\

We conclude this section with the following two necessary conditions for the existence of a minimal $\Gamma_n$-isometric dilation of a $\Gamma_n$-contraction. Before that we state a useful result from \cite{sourav14}.

\begin{prop}[\cite{sourav14}, Proposition 6.2 ]\label{spcoextn}
	Let $(R_1, \dots, R_{n-1},V)$ on $\mathcal{K}$ be a $\Gamma_n$-isometric dilation of a $\Gamma_n$-contraction $(S_1, \dots, S_{n-1},P)$ on $\mathcal{H}$. If $(R_1, \dots, R_{n-1},V)$ is minimal, then $(R_1^*, \dots, R_{n-1}^*,V^*)$ is a $\Gamma_n$-co-isometric extension of $(S_1^*,\dots, S_{n-1}^*,P^*)$.
\end{prop}

\begin{thm} \label{thm:BP001}
	Let $(S_1, \dots, S_{n-1},P)$ be a $\Gamma_n$-contraction defined on a Hilbert space $\mathcal{H}$ with $(A_1, \dots, A_{n-1})$ being the $\mathcal{F}_O$-tuple of $(S_1, \dots, S_{n-1},P)$. Let $(R_1, \dots, R_{n-1},V)$ on $\mathcal{K}$ be a minimal $\Gamma_n$-isometric dilation of the $\Gamma_n$-contraction $(S_1, \dots, S_{n-1},P)$. Then there exists a $\Gamma_n$-isometry $(W_1, \dots, W_{n-1},W)$ on $\mathcal{K}\ominus \mathcal{H}$ such that $(A_1, \dots, A_{n-1})=P_{\mathcal{D}_P}(W_1, \dots, W_{n-1})|_{\mathcal{D}_P}$.
\end{thm}
\begin{proof}
	Since $(R_1, \dots, R_{n-1},V)$ on $\mathcal{K}$ is a minimal $\Gamma_n$-isometric dilation of $(S_1, \dots, S_{n-1},P)$, without loss of assumption we can assume that 
	\[
	\mathcal{K}=\overline{\text{span}}\{R_1^{m_1}\dots R_{n-1}^{m_{n-1}}V^mh: h\in \mathcal{H} \text{ and }m_1, \dots, m_{n-1}, m\geq 0 \}.
	\]
	Also, by Theorem \ref{spcoextn} we have that $(R_1^*, \dots, R_{n-1}^*,V^*)|_{\mathcal{H}}=(S_1^*, \dots, S_{n-1}^*,P^*)$. With respect to the decomposition $\mathcal{K}=\mathcal{H}\oplus (\mathcal{K}\ominus \mathcal{H})$ suppose
	\[
	R_i= \begin{pmatrix}
	S_i & 0\\
	T_i & W_i
	\end{pmatrix}, \text{ i=1, \dots, n-1} \quad \text{ and } \quad V=\begin{pmatrix}
	P & 0\\
	T & W
	\end{pmatrix}.
	\]
	Since $V$ is an isometry, so $P^*P+T^*T= I_{\mathcal{H}}$ and $W^*W=I_{ \mathcal{K}\ominus\mathcal{H}}$. Define a map 
	\begingroup
	\allowdisplaybreaks
	\begin{align*}\label{eqV}
	\varphi: & \mathcal{D}_P \rightarrow \mathcal{K}\ominus \mathcal{H} \\
	&D_Px \mapsto Tx.
	\end{align*}
	\endgroup
	Clearly, $\varphi$ is an isometry. Since $(R_1, \dots, R_{n-1},V)$ is a $\Gamma_n$-isometry, so $R_i=R_{n-i}^*V$, that is,
	\[
	\begin{pmatrix}
	S_i & 0\\
	T_i & W_i
	\end{pmatrix}=\begin{pmatrix}
	S_{n-i}^* & T_{n-i}^*\\
	0 & W_{n-i}^*
	\end{pmatrix}\begin{pmatrix}
	P & 0\\
	T & W
	\end{pmatrix}
	\]
	This implies that $S_i-S_{n-i}^*P=T_{n-i}^*T$, $T_{n-i}^*W=0$ and $T_i=W_{n-i}^*T$. Since $(A_{1}, \dots, A_{n-1})$ is the $\mathcal{F}_O$-tuple of $(S_1, \dots, S_{n-1},P)$, so $D_PA_iD_P=T^*W_iT=D_P\varphi^*W_i\varphi D_P$. Then by the uniqueness of the $\mathcal{F}_O$-tuple, we have $A_i=\varphi^* W_i\varphi$. Therefore, $(A_1, \dots, A_{n-1})=P_{\mathcal{D}_P}(W_1, \dots, W_{n-1})|_{\mathcal{D}_P}$. Since $(W_1, \dots, W_{n-1},W)=(R_1, \dots, R_{n-1},V)|_{\mathcal{K}\ominus \mathcal{H}}$, so $(W_1, \dots, W_{n-1},W)$ is a $\Gamma_n$-contraction. Again since $W$ is an isometry, therefore, $(W_1, \dots, W_{n-1},W)$ is an $\Gamma_n$-isometry.
\end{proof}

\begin{thm}\label{necessary1}
	Let $(S_1, \dots, S_{n-1},P)$ be a $\Gamma_n$-contraction on a Hilbert space $\mathcal{H}$ with $(B_1, \dots, B_{n-1})$ being the $\mathcal{F}_O$-tuple of $(S_1^*, \dots, S_{n-1}^*,P^*)$. If $(R_1, \dots R_{n-1},V)$ on $\mathcal{K}=\mathcal{H}\oplus \ell^2(\mathcal{D}_P)$ is a $\Gamma_n$-isometric dilation of $(S_1, \dots, S_{n-1},P)$, then $(M_{B_1^*+zB_{n-1}}, \dots, M_{B_{n-1}^*+zB_{1}},M_z)$ on $H^2(\mathcal{D}_{P^*})$ is a pure $\Gamma_n$-isometry.
\end{thm}
\begin{proof}
	Clearly $V$ on $\mathcal K$ is the minimal isometric dilation of $P$ and thus $V^*|_{\mathcal{H}}=P^*$. Therefore, it follows trivially that $(R_1, \dots, R_{n-1},V)$ on $\mathcal{K}$ is a minimal $\Gamma_n$-isometric dilation of $(S_1, \dots, S_{n-1},P)$. So, we have $\mathcal{K}=\overline{\text{span}}\{V^nh: h\in \mathcal{H} \text{ and }n \geq 0\}$.  Now for $h\in \mathcal{H}$, 
	\[
	\|D_{V^*}^2h\|^2=\|h\|^2-\|V^*h\|^2=\|h\|^2-\|P^*h\|^2=\|D_{P^*}h\|^2 \quad \& \quad D_{V^*}^2V^nh=0 \quad \text{ for all } \quad n\geq 1.
	\]
	Then $\mathcal{D}_{V^*}=\overline{D_{V^*}\mathcal{K}}=\overline{D_{V^*}\mathcal{H}}$. Define a map $U:\mathcal{D}_{P^*}\to \mathcal{D}_{V^*}$ by $UD_{P^*}h=D_{V^*}h$ and extend to the closure continuously. Clearly $U$ is a unitary and $UD_{P^*}=D_{V^*}$. For $h_1,h_2\in \mathcal{H}$,
	\begingroup
	\allowdisplaybreaks
	\begin{align*}
		&\langle(S_i^*-S_{n-i}P^*)h_1,h_2 \rangle=\langle D_{P^*}B_iD_{P^*}h_1,h_2 \rangle\\
		\Rightarrow & \langle S_i^*h_1,h_2 \rangle-\langle P^*h_1, S_{n-i}^*h_2 \rangle= \langle B_{i}D_{P^*}h_1, D_{P^*}h_2 \rangle\\
		\Rightarrow & \langle R_i^*h_1,h_2 \rangle-\langle V^*h_1, R_{n-i}^*h_2 \rangle= \langle B_iU^*D_{V^*}h_1,U^*D_{V^*}h_2 \rangle \quad (\text{ by Theorem }\ref{spcoextn})\\
		\Rightarrow & \langle (R_i^*-R_{n-i}V^*)h_1,h_2 \rangle=\langle D_{V^*}(UB_iU^*)D_{V^*}h_1,h_2 \rangle.
	\end{align*}
	\endgroup
	Since $V$ is an isometry, so $D_{V^*}$ is a projection. Now for $h, h'\in \mathcal{H}$,
	\[
	\langle (R_i^*-R_{n-i}V^*)D_{V^*}h, D_{V^*}h' \rangle = \langle D_{V^*}(R_i^*-R_{n-i}V^*)D_{V^*}^2h, D_{V^*}h' \rangle.
	\]
	Therefore, $D_{V^*}(R_i^*-R_{n-i}V^*)D_{V^*}=(R_i^*-R_{n-i}V^*)$. Again for $h,h'\in \mathcal{H}$,
	\begingroup
	\allowdisplaybreaks
	\begin{align*}
		\langle (R_i^*-R_{n-i}V^*)D_{V^*}h, D_{V^*}h' \rangle&=\langle D_{V^*}(R_{i}^*-R_{n-i}V^*)D_{V^*}h, h' \rangle\\
		&=\langle (R_i^*-R_{n-i}V^*)h, h' \rangle\\
		&=\langle D_{V^*}(UB_iU^*)D_{V^*}h, h' \rangle\\
		&=\langle D_{V^*}(UB_iU^*)D_{V^*}^2h, D_{V^*}h' \rangle.
	\end{align*}
	\endgroup
	Therefore, $R_i^*-R_{n-i}V^*=D_{V^*}(UB_iU^*)D_{V^*}$. By uniqueness of the $\mathcal{F}_O$-tuple we have $(UB_1U^*, \\\dots, UB_{n-1}U^*)$ is the $\mathcal{F}_O$-tuple of $(R_1^*, \dots, R_{n-1}^*,V^*)$. Now apply Wold decomposition theorem to $(R_1, \dots, R_{n-1},V)$ on $\mathcal{K}$. Then the space $\mathcal{K}$ can be decomposed into $\mathcal{K}_1\oplus \mathcal{K}_2$ such that $(R_1, \dots, R_{n-1},V)|_{\mathcal{K}_1}$ is a pure $\Gamma_n$-isometry and $(R_1, \dots, R_{n-1},V)|_{\mathcal{K}_2}$ is a $\Gamma_n$-unitary. Suppose $(G_1, \dots, G_{n-1})$ is the $\mathcal{F}_O$-tuple of $(R_1^*, \dots, R_{n-1}^*,V^*)|_{\mathcal{K}_1}$. Then $(G_1\oplus 0, \dots, G_{n-1}\oplus 0)$ is the $\mathcal{F}_O$-tuple of $(R_1^*, \dots, R_{n-1}^*,V^*)$. Since $(R_1, \dots, R_{n-1},V)|_{\mathcal{K}_1}$ is a pure $\Gamma_n$-isometry, so is $(M_{G_1^*+zG_{n-1}}, \dots, M_{G_{n-1}^*+G_1},M_z)$ on $H^2(\mathcal{D}_{V^*})$ is a pure $\Gamma_n$-isometry. Therefore, $(M_{B_1^*+zB_{n-1}}, \dots, \\M_{B_{n-1}^*+zB_{1}},M_z)$ is a $\Gamma_n$-contraction and hence a pure $\Gamma_n$-isometry.
	
\end{proof}

\vspace{0.4cm}

\section{An abstract model for a class of c.n.u. $\Gamma_n$-contraction}\label{noncommutativemodel}

\vspace{0.4cm}

\noindent In this Section, we show an explicit construction of an operator model for a class of c.n.u. $\Gamma_n$-contractions which may or may not dilate to $\Gamma_n$-isometries. This precise class of c.n.u. $\Gamma_n$-contraction $(S_1,\dots, S_{n-1}, P)$ satisfies $S_i^*P=PS_i^*$ for each $i=1, \dots , n-1$. Once again the $\ft$-tuples of $(S_1,\dots, S_{n-1}, P)$ and $(S_1^*, \dots, S_{n-1}^*, P^*)$ play central role in this construction. Also, we provide a counterexample to show that such a model may not exist if we drop the condition that $S_i^*$ commutes with $P$ for each $i$. It is well known that for a contraction $P$ $(\text{or }P^*)$, $\{P^{*n}P^n : n\geq 1\}$ $(\text{or }\{P^nP^{*n} : n\geq 1\})$ is a non-increasing sequence of positive operators and it converges strongly to a positive operator. Suppose $\mathcal{A}$ is the strong limit of $\{P^{*n}P^n : n\geq 1\}$ and $\mathcal{A}_{*}$ is the strong limit of $\{P^nP^{*n} : n\geq 1\}$. Recall from Chapter 3 of \cite{K:C} the following map: 
\begingroup
\allowdisplaybreaks
\begin{align}\label{eqV}
	V: & \overline{ \text{Ran}}(\mathcal{A}) \rightarrow \overline{\text{Ran}}(\mathcal{A}) \notag\\
	&\mathcal{A}^{1/2}x \mapsto \mathcal{A}^{1/2}Px.
\end{align}
\endgroup

\begin{prop}[\cite{K:C}, Proposition 3.7]\label{Kcommute}
	The operator $\mathcal{A}^{1/2}\mathcal{A}_*\mathcal{A}^{1/2}$ on $\overline{ \text{Ran}}(\mathcal{A})$ commutes with $V$ and $V^*$.
\end{prop} 
The following result gives a necessary and sufficient condition under which a tuple of operators becomes the $\ft$-tuple of a $\Gamma_n$-contraction.
\begin{thm}[{\cite{B:P}, Theorem 2.1}]\label{BP1}
	A tuple of operators $(A_1, \dots, A_{n-1})$ defined on $\mathcal{D}_P$ is the $\ft$-tuple of a $\Gamma_n$-contraction $(S_1,\dots, S_{n-1}, P)$ if and only if $(A_1, \dots, A_{n-1})$ satisfy the following operator equations in $X_1, \dots, X_{n-1}$:
	\[
	D_PS_i = X_iD_P + X_{n-i}^*D_PP, \quad \text{ for } i= 1, \dots, n-1 \,.
	\]
\end{thm}
\begin{lem}\label{lem1}
	Let $(S_1,\dots, S_{n-1}, P)$ be a $\Gamma_n$-contraction on $\mathcal{H}$. Let $(A_1,\dots A_{n-1})$, $(B_1,\dots, B_{n-1})$ be the $\ft$-tuples of $(S_1,\dots, S_{n-1}, P)$ and $(S_1^*,\dots, S_{n-1}^*, P^*)$ respectively. Then for any $i = 1, \dots, n-1$
	\[
	B_i^*D_{P^*}\mathcal{A}^{1/2}|_{\overline{ \text{Ran}}(\mathcal{A})} + PP^*B_{n-i}D_{P^*}\mathcal{A}^{1/2}V = D_{P^*}S_i\mathcal{A}^{1/2}|_{\overline{ \text{Ran}}(\mathcal{A})}
	\] 
	and
	\[
	B_i^*D_{P^*}P^* + PP^*B_{n-i}D_{P^*} = D_{P^*}S_iP^*.
	\]
\end{lem}
\begin{proof}
	Suppose $x\,,y \in \mathcal{H}$. Now using $P^{*}\mathcal{A}P = \mathcal{A}$ we have
	\begingroup
	\allowdisplaybreaks
	\begin{align*}
		& \left\langle \left(B_i^*D_{P^*}\mathcal{A}^{1/2} + PP^*B_{n-i}D_{P^*}\mathcal{A}^{1/2}V\right)\mathcal{A}^{1/2}x, D_{P^*}y \right\rangle \\
		= & \left\langle D_{P^*}\left(B_i^*D_{P^*}\mathcal{A}^{1/2} + PP^*B_{n-i}D_{P^*}\mathcal{A}^{1/2}V\right)\mathcal{A}^{1/2}x, y \right\rangle \\
		= & \left\langle D_{P^*}B_{i}^*D_{P^*}\mathcal{A}x + PP^*D_{P^*}B_{n-i}D_{P^*}\mathcal{A}Px, y \right\rangle\\
		= & \left\langle \left(S_i - PS_{n-i}^* \right)\mathcal{A}x + PP^*\left(S_{n-i}^* - S_iP^*\right)\mathcal{A}Px, y \right\rangle\\
		= & \left\langle \left(S_i - PS_{n-i}^* \right)\mathcal{A}x + PS_{n-i}^*\mathcal{A}x - PP^*S_i\mathcal{A}x, y \right\rangle\\
		= & \left\langle D_{P^*}^2 S_i\mathcal{A}x, y \right\rangle\\
		= & \left\langle \left(D_{P^*}S_i\mathcal{A}^{1/2}\right)\mathcal{A}^{1/2}x, D_{P^*}y \right\rangle
	\end{align*}
	\endgroup
	and
	\begingroup
	\allowdisplaybreaks
	\begin{align*}
	& \left\langle \left(B_i^*D_{P^*}P^* + PP^*B_{n-i}D_{P^*}\right)x, D_{P^*}y \right\rangle \\
	= & \left\langle D_{P^*}\left(B_i^*D_{P^*}P^* + PP^*B_{n-i}D_{P^*}\right)x,y \right\rangle\\
	= & \left\langle \left(S_i - PS_{n-i}^*\right)P^*x + PP^*\left(S_{n-i}^* - S_{i}P^*\right)x,y  \right\rangle \\
	= & \left\langle D_{P^*}^2 S_{i}P^*x, y \right\rangle \\
	= & \left\langle D_{P^*}S_iP^*x, D_{P^*}y \right\rangle.
	\end{align*}
	\endgroup
	Therefore, for any $i=1,\dots, n-1$
	 \[
	 B_i^*D_{P^*}\mathcal{A}^{1/2}|_{\ov{Ran}\; \mathcal A} + PP^*B_{n-i}D_{P^*}\mathcal{A}^{1/2}V = D_{P^*}S_i\mathcal{A}^{1/2}
	 \]
	 and
	 \[
	 B_i^*D_{P^*}P^* + PP^*B_{n-i}D_{P^*} = D_{P^*}S_iP^*.
	 \]
\end{proof}

The following theorem provide model theory for a c.n.u. contraction.
\begin{thm}[Durszt, \cite{Durszt}]\label{durszt}
	A c.n.u. contraction $P$ on a Hilbert space $\mathcal{H}$
	is unitarily equivalent to a part of the direct sum of $M_z^*$ on $H^2(\mathcal{H})$ and $M_{e^{it}}^*$ on $L^2(\mathcal{H})$.
\end{thm}

\noindent Following the proof of Theorem \ref{durszt} we see that there is an isometry
\begin{equation}\label{eqW}
W:\mathcal{H} \rightarrow (H^2(\mathbb D)\otimes \mathcal{D}_{P}) \oplus (L^2(\mathbb T)\otimes \mathcal{D}_{P^*})
\end{equation}
that intertwines $P$ and $(M_z^*\otimes I_{\mathcal{D}_{P}}) \oplus (M_{e^{it}}^*\otimes I_{\mathcal{D}_{P^*}})$.
\[
WP = \left((M_z^*\otimes I_{\mathcal{D}_{P}}) \oplus (M_{e^{it}}^*\otimes I_{\mathcal{D}_{P^*}})\right)W.
\]
Needless to mention that $W$ has two components, say $W_1$ and $W_2$. We are going to recall these two components explicitly.
\begingroup
\allowdisplaybreaks
\begin{align}\label{eqW_1}
	W_1 : & \mathcal{H} \rightarrow H^2(\mathbb D)\otimes \mathcal{D}_P\notag\\
	& h \mapsto \sum\limits_{n=0}^{\infty}z^n \otimes D_PP^nh.
\end{align}
\endgroup
To define $W_2$, let us consider
\begingroup
\allowdisplaybreaks
\begin{align}\label{eqQ}
	Q: & \overline{\text{Ran}}(\mathcal{A}) \rightarrow \overline{\text{Ran}}(\mathcal{A})\notag \\
	&x \mapsto (I - \mathcal{A}^{1/2}\mathcal{A}_{*}\mathcal{A}^{1/2})^{1/2}x.
\end{align}
\endgroup
 By Proposition \ref{Kcommute}, $Q$ commutes with $V$ and $V^*$.
Suppose
\begin{equation}\label{eqH_0}
	\mathcal{H}_{0} = \{x \in \mathcal{H}: \mathcal{A}^{1/2}x \in \text{Ran}(Q)\}.
\end{equation}
Then following the proof of Theorem \ref{durszt}, ${\mathcal{H}}_0$ is a dense linear subspace of $\mathcal{H}$ and $\text{Ker}(Q)=\{0\}$. Hence $Q^{-1}$ exists on $\text{Ran}(Q)$. Define
\begingroup
\allowdisplaybreaks
\begin{align}\label{eqW_2}
	W_2 : &\mathcal{H}_0 \rightarrow L^2
	(\mathbb T) \otimes \mathcal{D}_{P^*}\notag \\
	& x \mapsto \sum_{n=-\infty}^{-1} z^n \otimes D_{P^*}\mathcal{A}^{1/2}Q^{-1}V^{*-n}\mathcal{A}^{1/2}x \notag \\
	& \qquad+ \sum_{n=0}^{\infty} z^n \otimes D_{P^*}\mathcal{A}^{1/2}Q^{-1}V^n\mathcal{A}^{1/2}x.
\end{align}
\endgroup
Now we present the model theorem which is the main result of this section.

\begin{thm}\label{noncommutemodel}
	Suppose $(S_1, \dots, S_{n-1}, P)$ on a Hilbert space $\mathcal{H}$ is a c.n.u. $\Gamma_n$-contraction with $S_i^*P = PS_i^*$ for all $i = 1, \dots, n-1$. Let $(A_1, \dots, A_{n-1})$ and $(B_1, \dots, B_{n-1})$ be the $\mathcal{F}_O$-tuples of $(S_1, \dots, S_{n-1}, P)$ and $(S_1^*, \dots, S_{n-1}^*, P^*)$ respectively. Consider $W=(W_1,W_2)$ as in \eqref{eqW}, \eqref{eqW_1} and \eqref{eqW_2}. Suppose $\mathcal{L}=\text{Ran}(W)$. Then 
	\begingroup
	\allowdisplaybreaks
	\begin{align*}
		S_i &\cong \left((I\otimes A_i + M_{z}^*\otimes A_{n-i}^*) \oplus (I \otimes B_i^* + M_{e^{it}}^* \otimes PP^*B_{n-i})\right)|_{\mathcal{L}}\\
		& \qquad\qquad\qquad\qquad\qquad\qquad\text{and}\\
		P & \cong \left((M_z^* \otimes I_{\mathcal{D}_P})\oplus (M_{e^{it}}^*\otimes I_{\mathcal{D}_{P^*}})\right)|_{\mathcal{L}}.
	\end{align*}
	\endgroup
	\end{thm}
\begin{proof}
	For any $h \in \mathcal{H}$, we have 
	\begingroup
	\allowdisplaybreaks
	\begin{align*}
	&\left(I \otimes A_i + M_{z}^* \otimes A_{n-i}^* \right)W_1h \\
	= &\left(I \otimes A_i + M_{z}^* \otimes A_{n-i}^* \right)\left(\sum_{k=0}^{\infty}z^k\otimes D_PP^kh\right)\\
	= &\sum_{k=0}^{\infty}z^k\otimes A_iD_PP^kh + \sum_{k=0}^{\infty}z^k\otimes A_{n-i}^*D_PP^{k+1}h\\
	= &\sum_{k=0}^{\infty}z^k\otimes \left(A_iD_P + A_{n-i}^*D_PP \right)P^kh\\
	= &\sum_{k=0}^{\infty}z^k\otimes D_PS_iP^k h \qquad\text{(by Lemma \ref{BP1})}\\
	= &\sum_{k=0}^{\infty}z^k\otimes D_PP^k S_ih\\
	= &W_1S_ih \,.
	\end{align*}
	\endgroup
	Therefore, $(I\otimes A_i + M_z^*\otimes A_{n-i}^*)W_1 = W_1S_i$.\\ 
	
	The linear subspace $\mathcal{H}_0$ is invariant under each $S_i$. Indeed, 
	for any $x \in \mathcal{H}_0$, $\mathcal{A}^{1/2}S_ix  = S_i\mathcal{A}^{1/2}x = S_iQu=QS_iu $ for some $u \in \overline{\text{Ran}}(\mathcal{A})$.
	The last equality follows from the fact that $S_i^*P=PS_i^*$. Now for $x \in \mathcal{H}_{0}$ and $z \in \mathbb T$ we have
	\begingroup
	\allowdisplaybreaks
	\begin{align*}
	& (I \otimes B_i^* + M_{z}^* \otimes PP^*B_{n-i})W_2x \\
	= & (I \otimes B_i^* + M_{z}^* \otimes PP^*B_{n-i})\Bigg( \sum_{n=-\infty}^{-1} z^n \otimes D_{P^*}\mathcal{A}^{1/2}Q^{-1}V^{*-n}\mathcal{A}^{1/2}x \\
	& \qquad\qquad\qquad \qquad\qquad\qquad + \sum_{n=0}^{\infty} z^n \otimes D_{P^*}\mathcal{A}^{1/2}Q^{-1}V^n\mathcal{A}^{1/2}x \Bigg)\\
	= & \sum_{n=-\infty}^{-1} z^n \otimes B_i^*D_{P^*}\mathcal{A}^{1/2}Q^{-1}V^{*-n}\mathcal{A}^{1/2}x + \sum_{n=0}^{\infty} z^n \otimes B_i^*D_{P^*}\mathcal{A}^{1/2}Q^{-1}V^n\mathcal{A}^{1/2}x\\
	& + \sum_{n=-\infty}^{-1} z^n \otimes PP^*B_{n-i}D_{P^*}\mathcal{A}^{1/2}Q^{-1}V^{*-n-1}\mathcal{A}^{1/2}x \\  & + \sum_{n=0}^{\infty} z^n \otimes PP^*B_{n-i}D_{P^*}\mathcal{A}^{1/2}Q^{-1}V^{n+1}\mathcal{A}^{1/2}x. 
	\end{align*}
	\endgroup
	Let us simplify the coefficients of the above expansion. We begin with the constant term. Note that
		\begingroup
	\allowdisplaybreaks
	\begin{align}\label{eqconst}
	& B_i^*D_{P^*}\mathcal{A}^{1/2}Q^{-1}\mathcal{A}^{1/2}x + PP^*B_{n-i}D_{P^*}\mathcal{A}^{1/2}Q^{-1}V\mathcal{A}^{1/2}x\notag\\
	=& B_i^*D_{P^*}\mathcal{A}^{1/2}Q^{-1}\mathcal{A}^{1/2}x + PP^*B_{n-i}D_{P^*}\mathcal{A}^{1/2}VQ^{-1}\mathcal{A}^{1/2}x\notag\\
	=&(B_i^*D_{P^*}\mathcal{A}^{1/2} + PP^*B_{n-i}D_{P^*}\mathcal{A}^{1/2}V)Q^{-1}\mathcal{A}^{1/2}x\notag\\
	=& D_{P^*}S_i\mathcal{A}^{1/2}Q^{-1}\mathcal{A}^{1/2}x \qquad\text{(by Lemma \ref{lem1})}\\
	= & D_{P^*}\mathcal{A}^{1/2}Q^{-1}\mathcal{A}^{1/2}S_ix. \notag
	\end{align}
	\endgroup
	In the last equality $Q^{-1}\mathcal{A}^{1/2}S_ix$ is make sense as $\mathcal{H}_0$ is invariant under $S_i$. We now consider the coefficient of $z^n$.	
	\begingroup
	\allowdisplaybreaks
	\begin{align*}
	& B_i^*D_{P^*}\mathcal{A}^{1/2}Q^{-1}V^n\mathcal{A}^{1/2}x + PP^*B_{n-i}D_{P^*}\mathcal{A}^{1/2}Q^{-1}V^{n+1}\mathcal{A}^{1/2}x\\
	= & (B_i^*D_{P^*}\mathcal{A}^{1/2} + PP^*B_{n-i}D_{P^*}\mathcal{A}^{1/2}V)Q^{-1}V^n\mathcal{A}^{1/2}x\\
	= & D_{P^*}S_i\mathcal{A}^{1/2}Q^{-1}V^n\mathcal{A}^{1/2}x \qquad \text{(by Lemma \ref{lem1})}\\
	= & D_{P^*}\mathcal{A}^{1/2}Q^{-1}V^n\mathcal{A}^{1/2}S_ix.
	\end{align*}
	\endgroup
	In a similar fashion we consider the coefficient of $z^{-n}$.	
	\begingroup
	\allowdisplaybreaks
	\begin{align*}
	& B_i^*D_{P^*}\mathcal{A}^{1/2}Q^{-1}V^{*n}\mathcal{A}^{1/2}x + PP^*B_{n-i}D_{P^*}\mathcal{A}^{1/2}Q^{-1}V^{*(n-1)}\mathcal{A}^{1/2}x\\
	= & (B_i^*D_{P^*}P^* + PP^*B_{n-i}D_{P^*})\mathcal{A}^{1/2}Q^{-1}V^{*(n-1)}\mathcal{A}^{1/2}\\
	= & D_{P^*}S_iP^*\mathcal{A}^{1/2}Q^{-1}V^{*(n-1)}\mathcal{A}^{1/2}x \qquad \text{(by Lemma \ref{lem1})}\\
	= & D_{P^*}\mathcal{A}^{1/2}Q^{-1}V^{*n}\mathcal{A}^{1/2}S_ix.
	\end{align*}
	\endgroup
	Hence, for all $x \in \mathcal{H}_0$,
	 \[
	 W_2S_ix = (I \otimes B_i^* + M_{e^{it}}^* \otimes PP^*B_{n-i})W_2 x .
	 \]
	 Therefore, if $W_0 = W|_{\mathcal{H}_0}:\mathcal{H}_0 \to \mathcal{L}$, then
	 \begingroup
	 \allowdisplaybreaks
	 \begin{align*}
	   & W_0S_ix \\
	 = & W_1S_ix \oplus W_2S_ix \\
	 = & (I \otimes A_i + M_z^* \otimes A_{n-i}^*)W_1x \oplus (I \otimes B_i^* + M_{e^{it}}^* \otimes PP^*B_{n-i})W_2 x \\
	 = & (I \otimes A_i + M_z^* \otimes A_{n-i}^* \oplus I \otimes B_i^* + M_{e^{it}}^* \otimes PP^*B_{n-i}) W_0 x,
	 \end{align*}
	 \endgroup
	 for all $x \in \mathcal{H}_0$. Since $W : \mathcal{H} \to \text{Ran}(W)$ extends $W_0 : \mathcal{H}_0  \to \mathcal{L}$, it follows that 
	 \[
	 WS_ix = \left((I \otimes A_i + M_z^* \otimes A_{n-i}^*) \oplus (I \otimes B_i^* + M_{e^{it}}^* \otimes PP^*B_{n-i}) \right)Wx
	 \]
	 for all $x \in \mathcal{H}$. This implies that $\mathcal{L}(=\text{Ran}(W))$ is an invariant subspace for $(I \otimes A_i + M_z^* \otimes A_{n-i}^*) \oplus (I \otimes B_i^* + M_{e^{it}}^* \otimes PP^*B_{n-i})$. Therefore, 
	 \begingroup
	 \allowdisplaybreaks
	 \begin{align*}
	 	S_i
	 	 =& W^*\big(((I \otimes A_i + M_z^* \otimes A_{n-i}^*) \oplus (I \otimes B_i^* + M_{e^{it}}^* \otimes PP^*B_{n-i}))|_{\mathcal{L}}\big)W\\
	 	 \cong & \big((I \otimes A_i + M_z^* \otimes A_{n-i}^*) \oplus (I \otimes B_i^* + M_{e^{it}}^* \otimes PP^*B_{n-i})\big)|_{\mathcal{L}}.
	 \end{align*}
	 \endgroup
	 By Theorem \ref{durszt}, 
	 \[
	 P \cong \big((M_z^* \otimes I_{\mathcal{D}_P})\oplus (M_{e^{it}}\otimes I_{\mathcal{D}_{P^*}})\big)|_{\mathcal{L}}.
	 \]
	 This completes the proof.
\end{proof}

It is evident from the above construction that the model need not a commutative one unless we assume commutativity conditions on the $\ft$-tuples and their adjoints. Also, the conclusion of the above theorem may fail if we drop the assumption that $S_i^*$ commutes with $P$. Before going to present a counter example, let us recall famous Ando's dilation.

\begin{thm}[\textbf{Ando}, \cite{Ando}]\label{thm2.1}
	For every commuting pair of contractions $(T_1, T_2)$ on a Hilbert space $\mathcal{H}$ there exists a commuting pair of unitaries $(U_1, U_2)$ on a Hilbert space $\mathcal{K}$ containing $\mathcal{H}$ as a subspace such that 
	\[
	T_1^{n_1}T_2^{n_2} = P_{\mathcal{H}}U_1^{n_1}U_2^{n_2}|_{\mathcal{H}}\, \qquad \text{ for } n_1, n_2 \geq 0.
	\]
\end{thm}
By Theorem \ref{thm2.1}, a commuting pair of contractions $(T_1,T_2)$ on $\mathcal{H}$ dilates to a commuting pair of unitaries $(U_1,U_2)$ on a Hilbert space $\mathcal{K}$ containing $\mathcal{H}$. Therefore, the $n$-tuple of commuting contractions $(I_{\mathcal{H}}, \dots, I_{\mathcal{H}}, T_1,T_2)$ on $\mathcal{H}$ dilates to the commuting $n$-tuple of unitaries $(I_{\mathcal{K}}, \dots, I_{\mathcal{K}},U_1,U_2)$ on $\mathcal{K}$. Thus \textit{von Neumann inequality} holds for $(I_\mathcal{H}, \dots, I_{\mathcal{H}}, T_1,T_2)$, that is,
\[
\|p(I_{\mathcal{H}}, \dots, I_{\mathcal{H}}, T_1,T_2)\| \leq \|p\|_{\infty, \overline{\mathbb{D}^n}}, \text{ for any }p\in \mathbb{C}[z_1, \dots, z_n].
\]

For $n \geq 2$,  the symmetrization map in $n$-complex variables $z = (z_1,\dots,z_n)$ is the following proper holomorphic map
\[
\pi_n(z) = (s_1(z), \dots, s_{n-1}(z), p(z))
\]
where
\[
s_i(z)=\sum\limits_{1\leq k_1 <k_2 \dots <k_i\leq n}z_{k_1}\dots z_{k_i}\,,\quad i= 1, \dots, n-1  \text{ and } p(z) = \prod_{i=1}^{n}z_i \,.
\]
\begin{lem}\label{lem2.3}
	For a commuting pair of contractions $(T_1, T_2)$ on a Hilbert space $\mathcal{H}$, the symmetrization $\pi_n(I_{\mathcal{H}}, \dots, I_{\mathcal{H}}, T_1, T_2)$ is a $\Gamma_n$-contraction on $\mathcal{H}$.
\end{lem}
\begin{proof}
	For any holomorphic polynomial $p$ in $n$-variables we have
	\[
	\|p\circ\pi_n(I_{\mathcal{H}}, \dots, I_{\mathcal{H}}, T_1, T_2)\| \leq \|p\circ\pi_n\|_{\infty, \overline{\mathbb{D}}^n}= \|p\|_{\infty, \pi_n(\overline{\mathbb{D}}^n)} \leq \|p\|_{\infty, \Gamma_n}.
	\]
	Then, $\pi_n(I_{\mathcal{H}}, \dots, I_{\mathcal{H}}, T_1, T_2)$ is a $\Gamma_n$-contraction on $\mathcal{H}$.
	
\end{proof}
We are now ready to present a counterexample.
\begin{eg}\label{example1}
	We consider $\mathcal{H} = \ell^2$, where
	\[
	\ell^2 = \left\{\{x_n\}_n : x_n \in \mathbb{C} \text{ and } \sum_{n=1}^{\infty}|x_n|^2 < \infty\right\}.
	\]
	Consider the operator $T : \ell^2 \rightarrow \ell^2$ defined by 
	\[
	T(x_1, x_2, x_3, \dots) = (0, \alpha x_1, x_2, x_3, \dots), \text{ for some }\alpha \in (0,1).
	\]
	Then by Lemma \ref{lem2.3}, $\pi_n(I_{\mathcal{H}}, \dots, I_{\mathcal{H}}, T, T)$ is a $\Gamma_n$-contraction. Let $(S_1, \dots, S_{n-1},P)$ be the symmetrization of $(I_{\mathcal{H}}, \dots, I_{\mathcal{H}}, T, T)$. Then $S_1=(n-2)I + 2T $, $S_{n-1}=2T+(n-2)T^2$ and $P=T^2$. Note that $S_1^*P \neq PS_1^*$. One can easily check that
	\begingroup
	\allowdisplaybreaks
	\begin{align*}
		&D_P(x_1,x_2,x_3,\dots)=\left( \sqrt{(1 - \alpha^2)}x_1,0,0,\dots \right),\\
		&D_{P^*}(x_1,x_2,x_3,\dots) =\left( x_1,x_2,\sqrt{(1-\alpha^2)}x_3,0,0,\dots \right),\\
		&\mathcal{A}^{1/2}(x_1,x_2,x_3,\dots) = \left( \alpha x_1,x_2,x_3,\dots \right),\\
		&\mathcal{A}_*^{1/2}(x_1,x_2,x_3,\dots) = (0,0,0,\dots),\\
		&Q(x_1,x_2,x_3,\dots)=(x_1,x_2,x_3,\dots) \text{ and }\mathcal{H}_0=\ell^2,
	\end{align*}
	\endgroup
	where $\mathcal{A}$ $(\text{or }\mathcal{A}_*)$ is the strong limit of $\{P^{*n}P^n:n\geq 1\}$ $(\text{or }\{P^nP^{*n}:n\geq 1\})$, $Q$ is in \eqref{eqQ} and $\mathcal{H}_0$ is described in \eqref{eqH_0}. 
	Again it can easily be checked that $S_1^*-S_{n-1}P^*=D_{P^*}B_1D_{P^*}$ and $S_{n-1}^* - S_1P^*=D_{P^*}B_{n-1}D_{P^*}$, where 
	\begingroup
	\allowdisplaybreaks
	\begin{align*}
	&B_{1}(x_1,x_2,x_3,\dots)\\
	 = & \left( (n-2)x_1+2\alpha x_2, (n-2)x_2+2\sqrt{1-\alpha^2}x_3, (n-2)x_3, 0,0,\dots \right)\\
	\text{and }&\\
	& B_{n-1}(x_1,x_2,x_3,\dots)\\
	= & \left( 2\alpha x_2, 2 \sqrt{1-\alpha^2}x_3,0,0,\dots \right).
	\end{align*}
	\endgroup
	By \eqref{eqconst}, the constant term of $\left( I\otimes B_1^*+ M_z^*\otimes PP^*B_{n-1} \right)W_2$ is $D_{P^*}S_1\mathcal{A}$. Now 
	\begingroup
	\allowdisplaybreaks
	\begin{align*}
		&D_{P^*}S_1\mathcal{A}(x_1,x_2,x_3,\dots)\\
		=& D_{P^*}S_1\left( \alpha^2 x_1, x_2,x_3,\dots \right)\\
		=&(n-2)D_{P^*}\left( \alpha^2 x_1,x_2,x_3,\dots \right)+2D_{P^*}\left( 0, \alpha^3x_1,x_2,x_3,\dots \right)\\
		=& (n-2)\left( \alpha^2 x_1,x_2,\sqrt{1-\alpha^2}x_3,0,0,\dots \right)+2\left( 0, \alpha^3x_1,\sqrt{1-\alpha^2}x_2,0,\dots \right).
	\end{align*}
	\endgroup
	Again the constant term of $W_2S_1$ is $D_{P^*}\mathcal{A}S_1$. Now
	\begingroup
	\allowdisplaybreaks
	\begin{align*}
		&D_{P^*}\mathcal{A}S_1(x_1,x_2,x_3,\dots)\\
		=&(n-2)D_{P^*}\mathcal{A}(x_1,x_2,x_3,\dots) + 2 D_{P^*}\mathcal{A}\left( 0, \alpha x_1,x_2,x_3,\dots \right)\\
		=&(n-2)D_{P^*}\left( \alpha^2x_1,x_2,x_3,\dots \right)+2D_{P^*}\left( 0, \alpha x_1,x_2,x_3, \dots \right)\\
		=&(n-2)\left( \alpha^2x_1,x_2,\sqrt{1-\alpha^2}x_3,0,\dots \right)+2\left( 0, \alpha x_1,\sqrt{1-\alpha^2}x_2, 0,\dots \right).
	\end{align*}
	\endgroup
	Therefore, the constant term of $\left( I\otimes B_1^*+ M_z^*\otimes PP^*B_{n-1} \right)W_2$ and the constant term of $W_2S_1$ are different. Thus the conclusion of Theorem \ref{noncommutemodel} is not true for the $\Gamma_n$-contraction $(S_1, \dots,S_{n-1},P)$.
	
\end{eg}

\vspace{0.4cm}

\section{A complete unitary invariant for a class of c.n.u. $\Gamma_n$-contractions} \label{completunitaryinvariant}
	
	\vspace{0.4cm}

\noindent Let us recall from (\ref{char}) the characteristic function of a contraction $P \in \mathcal B(\HS)$ which is defined on the set $\Lambda_P = \{ z\in \mathbb C : (I - zP^*) \text{ is invertible}\}$.
\begin{equation}\label{char-001}
\Theta_P(z) = [-P + zD_{P^*}(I - zP^*)^{-1}D_P]|_{\mathcal{D}_P}\,.
\end{equation} 

\begin{defn}\label{coincide}
	Suppose $P$ and $P'$ are two contractions on $\mathcal{H}$ and $\mathcal{H'}$ respectively. Then we say that the characteristic functions of $P$ and $P'$ coincide if there are unitaries $u: \mathcal{D}_P \to
	\mathcal{D}_{P'}$ and $u_{*}: \mathcal{D}_{P^*} \to
	\mathcal{D}_{{P'}^*}$ such that $u_*\Theta_P(z)=\Theta_{P'}(z)u$ for all $z \in \mathbb D$.
\end{defn}
 The following remarkable result due to Sz.-Nagy and Foias shows that the characteristic function is a complete unitary invariant for c.n.u. contractions.
\begin{thm}[Nagy-Foias, \cite{Nagy}]\label{nf}
	Two c.n.u. contractions are unitarily equivalent
	if and only if their characteristic functions coincide.
\end{thm}
Here for our purpose we shall follow the terminologies that we define below. 
\begin{defn}\label{chartuple}
	Let $(S_1, \dots, S_{n-1},P)$ be a $\Gamma_n$-contraction on a Hilbert space $\mathcal{H}$. Suppose $(A_1,\dots, A_{n-1})$ is the $\ft$-tuple of $(S_1, \dots, S_{n-1},P)$ and $\Theta_P$ is the characteristic function of $P$. Then the tuple $(A_1,\dots, A_{n-1}, \Theta_P)$ is called the characteristic tuple for $(S_1, \dots, S_{n-1},P)$.
\end{defn}

\begin{defn}\label{unieq}
Let $(S_1, \dots, S_{n-1},P)$ and $(S_1', \dots, S_{n-1}',P')$ are two $\Gamma_n$-contractions defined on $\mathcal{H}$ and $\mathcal{H}'$ respectively. Suppose $(A_1, \dots, A_{n-1})$ and $(A_1', \dots, A_{n-1}')$ are the $\ft$-tuples of $(S_1, \dots, S_{n-1},P)$ and $(S_1', \dots, S_{n-1}', P')$ respectively while
	$(B_1, \dots, B_{n-1})$ and $(B_1', \dots, B_{n-1}')$ are the $\ft$-tuples of $(S_1^*,\dots,S_{n-1}^*,P^*)$ and $(S_1'^*,\dots,S_{n-1}'^*,P'^*)$ respectively. Then
	\begin{enumerate}
		\item[(i)] the $\ft$-tuples $(B_1, \dots, B_{n-1})$ and $(B_1', \dots, B_{n-1}')$ are \textit{C-unitarily equivalent} if $u_{*}B_i = B_i'u_{*}$\,, for $i = 1,\dots,n-1$, where $u_{*}:\mathcal{D}_{P^*} \rightarrow \mathcal{D}_{P'^*}$ is unitary that is involved in the coincidence of the characteristic functions of $P$ and $P'$;
		\item[(ii)] the characteristic tuples $(A_1, \dots,A_{n-1},\Theta_P)$ and $(A_1',\dots,A_{n-1}',  \Theta_{P'})$ are \textit{C-unitarily equivalent} if the characteristic functions $\Theta_P$ of $P$ and $\Theta_{P'}$ of $P'$ coincide and $(A_1,\dots, A_{n-1})$ is unitarily equivalent to $(A_1', \dots, A_{n-1}')$ by the unitary $u : \mathcal{D}_P \rightarrow \mathcal{D}_{P'}$ that is involved in the coincidence of $\Theta_P$ and $\Theta_{P'}$.
	\end{enumerate}
\end{defn}

We now present a \textit{complete unitary invariant} for c.n.u. $\Gamma_n$-contractions $(S_1, \dots, S_{n-1},P)$ such that $S_iP^*=P^*S_i$ for all $i=1, \dots ,n-1$.

\begin{thm}\label{unitaryinv}
	Let $(S_1,\dots,S_{n-1},P)$ and $(S_1',\dots,S_{n-1}',P')$ be two c.n.u. $\Gamma_n$-contractions acting on $\mathcal{H}$ and $\mathcal{H'}$ respectively with $S_i^*P = PS_i^*$ and $S_i^{\prime *}P^{\prime} = P^{\prime}S_i^{\prime *}$ for all $i$, 
	defined on $\mathcal{H}$ and $\mathcal{H'}$ respectively.
	Suppose $(A_1, \dots, A_{n-1})$ and $(A_1^{\prime}, \dots, A_{n-1}^{\prime})$ are the $\ft$-tuples of $(S_1, \dots, S_{n-1},P)$ and $(S_1',\dots,S_{n-1}',P')$ respectively while $(B_1,\dots,B_{n-1})$ and $(B_1',\dots,B_{n-1}')$ are the $\ft$-tuples of $(S_1^*,\dots,S_{n-1}^*,P^*)$ and $(S_1'^*,\dots,S_{n-1}'^*,P'^*)$ respectively. Then $(S_1,\dots, S_{n-1},P)$ is unitarily equivalent to $(S_1', \dots,S_{n-1}',P')$ if and only if
	the characteristic tuples of $(S_1, \dots, S_{n-1},P)$ and $(S_1', \dots, S_{n-1}',P')$ are C-unitarily equivalent and the $\ft$-tuples $(B_1, \dots, B_{n-1})$, $(B_1', \dots, B_{n-1}')$ are C-unitarily equivalent.
\end{thm}
\begin{proof}
	Suppose $(S_1,\dots,S_{n-1},P)$ and $(S_1',\dots,S_{n-1}',P')$ are unitarily equivalent and $U : \mathcal{H} \rightarrow \mathcal{H}'$ is a unitary such that $US_i = S_i'U$ for each $i$ and $UP = P'U$. Since both $P\text{ and }P'$ are c.n.u. contractions, then the characteristic functions $\Theta_P$ and $\Theta_{P'}$ coincide. The unitaries that are involved in the coincidence of $\Theta_P$ and $\Theta_{{P'}}$ are nothing but the restrictions of $U$ to $\mathcal{D}_P$ and $\mathcal{D}_{P^*}$. Let $U_1 = U|_{\mathcal{D}_P}$ and $U_2 = U|_{\mathcal{D}_{P^*}}$. Since $UD_P = D_{P'}U$ and $UD_{P^*} = D_{P'^*}U$, so $U_1 \in \mathcal{B}(\mathcal{D}_P, \mathcal{D}_{P'})$ and $U_2 \in \mathcal{B}(\mathcal{D}_{P^*}, \mathcal{D}_{P'^*}) $. Then for any $h\in \mathcal{H}'$ and for each $i$ we have,
	\begingroup
	\allowdisplaybreaks
	\begin{align*}
	D_{P'}U_1A_iU_1^*D_{P'}h = D_{P'}UA_iU^*D_{P'}h &= UD_PA_iD_PU^*h\\
	&= U\left(S_i - S_{n-i}^*P\right)U^*h \\
	& = \left(S_i' - S_{n-i}'^*P'\right)h\\
	& = \left(D_{P'}A_i'D_{P'}\right)h.
	\end{align*}
	\endgroup
	Therefore, if $J=U_1A_iU_1^*-A_i'$ then $J:\mathcal{D}_{P'}\to \mathcal{D}_{P'}$ and $D_{P'}JD_{P'}=0$. Now
	\[
	\langle JD_{P'}h_1,D_{P'}h_2 \rangle=\langle D_{P'}JD_{P'}h_1,h_2 \rangle=0 \text{ for all }h_1, h_2 \in \mathcal{H}'.
	\]
	This shows that $J=0$ and hence $U_1A_iU_1^*=A_i'$. Similarly, $U_2B_iU_2^*=B_i'$.\\
	
	Conversely, suppose the characteristic tuples of $(S_1, \dots, S_{n-1},P)$ and $(S_1', \dots, S_{n-1}',P')$ are C-unitarily equivalent and the $\ft$-tuples $(B_1,\dots, B_{n-1})$ and $(B_1', \dots, B_{n-1}')$ are C-unitarily equivalent. Then by Definition \ref{unieq}, there exist unitaries $u : \mathcal{D}_P \rightarrow \mathcal{D}_{P'}$ and $u_* : \mathcal{D}_{P^*} \rightarrow \mathcal{D}_{P'^*}$ such that for each $i$ and for all $z \in \mathbb{D}$
	\[
	uA_i = A_i'u,\,  u_*B_i = B_i'u_* \text{ and } u_*\Theta_P(z) = \Theta_{{P'}}(z)u\,.
	\]
	Following the proof of Theorem \ref{nf}, there exists a unitary $U : \mathcal{H} \rightarrow \mathcal{H}'$ such that $U|_{\mathcal{D}_P} = u$, $U|_{\mathcal{D}_{P^*}} = u_*$ and $UP = P'U$. Consider
	\begingroup
	\allowdisplaybreaks
	\begin{align*}
	U_1 : & H^2(\mathbb{D}) \otimes \mathcal{D}_P \rightarrow H^2(\mathbb{D})\otimes\mathcal{D}_{P'} \\
	& \sum\limits_{n=0}^{\infty}z^n \otimes x_n \mapsto \sum\limits_{n=0}^{\infty}z^n \otimes Ux_n
	\end{align*}
	\endgroup
	and 
	\begingroup
	\allowdisplaybreaks
	\begin{align*}
	U_2 : & L^2(\mathbb{T}) \otimes \mathcal{D}_{P^*} \rightarrow L^2(\mathbb{T})\otimes\mathcal{D}_{P'^*} \\
	& \sum\limits_{n=-\infty}^{\infty}e^{int} \otimes y_n \mapsto \sum\limits_{n=-\infty}^{\infty}e^{int} \otimes Uy_n.
	\end{align*}
	\endgroup
	In the similar fashion we can define $\mathcal{A}'$, $V',\, Q',\, \mathcal{H}'_0,\, W_1', \, W_2'$ for $P'$ as we have defined for $P$ in the previous section.
	Using the fact $UP=P'U$, one can easily check that $U\mathcal{H}_0 \subseteq \mathcal{H}_0'$.
	By Theorem \ref{durszt}, for $P'$ on
	$\mathcal{H}'$ there exists an isometry $W' : \mathcal{H}'
	\rightarrow H^2(\mathbb D)\otimes \mathcal{D}_{P'} \oplus
	L^2(\mathbb T) \otimes \mathcal{D}_{P'^*}$ such that
	\[
	W'P' = \left( M_z^* \otimes I_{\mathcal{D}_{P'}} \oplus M_{e^{it}}^*
	\otimes I_{\mathcal{D}_{P'^*}} \right)W'.
	\]
	Now for any $x \in \mathcal{H}_0$,
	\begingroup
	\allowdisplaybreaks
	\begin{align*}
		U_1W_1x  = \widetilde{U}\left(\sum_{n=0}^{\infty}z^n \otimes D_PP^nx\right)  &= \sum_{n=0}^{\infty} z^n \otimes UD_PP^nx \\
		& = \sum_{n=0}^{\infty} z^n \otimes D_{P'}P'^nUx \\
		& = W_1'Ux
	\end{align*}
	\endgroup
	and
	\begingroup
	\allowdisplaybreaks
	\begin{align*}
		U_2W_2x  = &\, \widetilde{U}_* \left( \sum_{n=-\infty}^{-1}z^n \otimes D_{P^*}\mathcal{A}^{1/2}Q^{-1}V^{*-n}\mathcal{A}^{1/2}x \right)\\
		& + \widetilde{U}_* \left( \sum_{n=0}^{\infty}z^n \otimes D_{P^*}\mathcal{A}^{1/2}Q^{-1}V^{n}\mathcal{A}^{1/2}x \right)\\
		= & \left( \sum_{n=-\infty}^{-1}z^n \otimes UD_{P^*}\mathcal{A}^{1/2}Q^{-1}V^{*-n}\mathcal{A}^{1/2}x \right) \\
		& + \left( \sum_{n=0}^{\infty}z^n \otimes UD_{P^*}\mathcal{A}^{1/2}Q^{-1}V^{n}\mathcal{A}^{1/2}x \right) \\
		= & \left( \sum_{n=-\infty}^{-1}z^n \otimes D_{P'^*}\mathcal{A}'^{1/2}Q'^{-1}V'^{*-n}\mathcal{A}'^{1/2}Ux \right) \\
		& + \left( \sum_{n=0}^{\infty}z^n \otimes D_{P'^*}\mathcal{A}'^{1/2}Q'^{-1}V'^{n}\mathcal{A}'^{1/2}Ux \right) \\
		= & W_2'Ux\,.
	\end{align*}
	\endgroup
	In the second last equality $Q'^{-1}V'^{*-n}\mathcal{A}'^{1/2}Ux$ and $Q'^{-1}V'^n\mathcal{A}'^{1/2}Ux$ are make sense as $U\mathcal{H}_0\subseteq \mathcal{H}_0'$.
	Therefore, for every $x\in\mathcal{H}_0$, 
	\[
	\left(U_1 \oplus U_2 \right)\left( W_1 \oplus W_2 \right)x = \left( W_1' \oplus W_2' \right)Ux\,.
	\]
	Since $U$ is a unitary from $\mathcal{H}_0$ into $\mathcal{H}_0'$ and $\overline{\mathcal{H}}_0 = \mathcal{H}$, so $\overline{ U\mathcal{H}}_0=U\overline{\mathcal{H}}_0=\mathcal{H}'$. Hence $U_1 \oplus U_2$ is a unitary from $\text{Ran}(W_1\oplus W_2)$ onto $\text{Ran}(W_1' \oplus W_2')$. Now by the definitions of $U_1$ and $U_2$ one can easily check that 
	\begingroup
	\allowdisplaybreaks
	\begin{align*}
		& \left(U_1 \oplus U_2 \right)\left((I \otimes A_i + M_z^* \otimes A_{n-i}^*) \oplus (I \otimes B_i^* + M_{e^{it}}^* \otimes PP^*B_{n-i}) \right) \\
		= & \left(\left( I \otimes A_i' + M_z^* \otimes A_{n-i}'^* \right) \oplus \left( I \otimes B_i'^* + M_{e^{it}}^* \otimes P'P'^*B_{n-i}' \right)\right)\left( U_1 \oplus U_2 \right)
	\end{align*}
	\endgroup
	and
	\begingroup
	\allowdisplaybreaks
	\begin{align*}
		& \left(U_1 \oplus U_2 \right) \left( M_z^* \otimes I_{\mathcal{D}_P} \oplus M_{e^{it}}^* \otimes I_{\mathcal{D}_{P^*}} \right) \\
		= & \left( M_z^* \otimes I_{\mathcal{D}_{P'}} \oplus M_{e^{it}}^* \otimes I_{\mathcal{D}_{P'^*}} \right)\left(U_1 \oplus U_2 \right).
	\end{align*}
	\endgroup
	Therefore, $(S_1, \dots, S_{n-1}, P)$ and $(S_1', \dots, S_{n-1}', P')$ are unitarily equivalent. The proof is now complete.
\end{proof}

\vspace{0.3cm}

\section{Beurling-Lax-Halmos type representation: an alternative proof} \label{lastSec}

\vspace{0.4cm}

\noindent A Beurling-Lax-Halmos type reprresentation theorem for invariant subspaces of a pure $\Gamma_n$-isometry was obtained in \cite{S:B}. Here we provide an alternative proof to that result.
	\begin{thm} [\cite{S:B}, Theorem 5.3]
		Let $\mathcal{M}$ be a non-zero closed linear subspace of $H^2(\mathcal{E})$. Then $\mathcal{M}$ is an invariant subspace of a pure $\Gamma_n$-isometry $(M_{A_1^*+zA_{n-1}}, \dots, M_{A_{n-1}^*+zA_1},M_z)$ on $H^2(\mathcal{E})$ if and only if there exist unique $B_1, \dots, B_{n-1}\in H^2(\mathcal{E}_*)$ such that $(M_{B_1+zB_{n-1}^*}, \dots, M_{B_{n-1}+zB_1^*},M_z)$ on $H^2(\mathcal{E}_*)$ is a pure $\Gamma_n$-isometry and for all $z\in \D$ ,
		\[
		(A_i^*+zA_{n-i})\Theta(z)=\Theta(z)(B_i+zB_{n-i}^*), \quad \text{ for }i= 1, \dots, n-1,
		\]
		where $(\mathcal{E}_*,\mathcal{E},\Theta)$ is the Beurling-Lax-Halmos representation of $\mathcal{M}$.
	\end{thm}
\begin{proof}
	We prove only the forward direction and a proof for the reverse direction is similar. Let $\mathcal{M}$ be invariant under $(M_{A_1^*+zA_{n-1}}, \dots, M_{A_{n-1}^*+zA_1},M_z)$. Then $\mathcal{M}$ is invariant under $M_z$. Let $\mathcal{M}=M_{\Theta}H^2(\mathcal{E}_*)$ be the Beurling-Lax-Halmos representation of $\mathcal{M}$, where $(\mathcal{E}_*,\mathcal{E},\Theta)$ is an inner multiplier. Since $\mathcal{M}$ is $M_{A_i^*+A_{n-i}z}$-invariant, we have 
	\[
	M_{A_i^*+zA_{n-i}}M_{\Theta}H^2(\mathcal{E}_*)\subseteq M_{\Theta}H^2(\mathcal{E}_*), \quad \text{ for all } i=1, \dots, n-1.
	\]
	Then there exist $G_1, \dots, G_n$ in $\mathcal B(H^2(\mathcal{E}_*))$ such that 
	\begin{equation}\label{(1)}
		M_{A_i^*+zA_{n-i}}M_{\Theta}=M_{\Theta}G_i.
	\end{equation}
	Since $M_{\Theta}$ is an isometry, $G_i$'s are well defined and unique. Now from \eqref{(1)} we have $G_i=M_{\Theta}^*M_{A_{i}^*+zA_{n-i}}M_{\Theta}$ which further implies that 
	\begingroup
	\allowdisplaybreaks
	\begin{align*}
		 M_zG_i&=M_zM_{\Theta}^*M_{A_i^*+zA_{n-i}}M_{\Theta}\\
		&=M_{\Theta}^*M_{A_i^*+zA_{n-i}}M_{\Theta}M_z\\
		&=G_iM_z \text{ for }i= 1, \dots, n-1.
	\end{align*}
	\endgroup
	Therefore, $G_i=M_{\Phi_i}$ and $G_{n-i}=M_{\Phi_{n-i}}$ for some $\Phi_i, \Phi_{n-i}\in H^{\infty}(\mathcal{B}(\mathcal{E}_*))$. From \eqref{(1)} we have $M_{\Phi_i}=M_{\Theta}M_{A_i^*+zA_{n-i}}M_{\Theta}$. Now
	\[
	M_z^*M_{\Phi_i}=M_{\Theta}^*M_{A_{n-i}^*+zA_{i}}M_{\Theta}=M_{\Phi_{n-i}}^*,
	\] 
	that is,
	$
	M_{\Phi_{n-i}}=M_{\Phi_i}^*M_z$. Similarly we have $M_{\Phi_i}=M_{\Phi_{n-i}}^*M_z$.
	Considering the power series expressions of $\Phi_i$ and $\Phi_{n-i}$ and using $M_{\Phi_i}=M_{\Phi_{n-i}}^*M_z$ and $M_{\Phi_{n-i}}=M_{\Phi_i}^*M_z$, we get $\Phi_i(z)=B_i+zB_{n-i}^*$ and $\Phi_{n-i}(z)=B_{n-i}+zB_{i}^*$, for some $B_i, B_{n-i}\in \mathcal{B}(\mathcal{E}_*)$. Uniqueness of $B_i$ and $B_{n-i}$ follows from the uniqueness of $G_i$ and $G_{n-i}$. It is clear that $M_zM_{B_i+zB_{n-i}^*}=M_{B_i+zB_{n-i}^*}M_z$. Now 
	\begingroup
	\allowdisplaybreaks
	\begin{align*}
		M_{B_i+zB_{n-i}^*}M_{B_j+zB_{n-j}^*}&=M_{\Theta}^*M_{A_i^*+zA_{n-i}}M_{\Theta}M_{\Theta}^*M_{A_j^*+zA_{n-j}}M_{\Theta}\\
		&=M_{\Theta}^*M_{A_i^*+zA_{n-i}}M_{A_j^*+zA_{n-j}}M_{\Theta}\\
		&=M_{\Theta}^*M_{A_j^*+zA_{n-j}}M_{A_i^*+zA_{n-i}}M_{\Theta}\\
		&=M_{\Theta}^*M_{A_j^*+zA_{n-j}}M_{\Theta}M_{\Theta}^*M_{A_i^*+zA_{n-i}}M_{\Theta}\\
		&=M_{B_j+zB_{n-j}^*}M_{B_i+zB_{n-i}^*}.
	\end{align*}
	\endgroup
	Therefore, $\left(M_{B_1+zB_{n-1}^*}, \dots, M_{B_{n-i}+zB_{i}^*},M_z\right)$ is a commuting tuple. So, for any $p\in \mathbb{C}[z_1, \dots, z_{n}]$ we have 
	\begingroup
	\allowdisplaybreaks
	\begin{align*}
		\left\|p\left(M_{B_1+zB_{n-1}^*}, \dots, M_{B_{n-i}+zB_{i}^*},M_z\right)\right\|
		& =\left\|M_{\Theta}^*p\left(M_{A_1^*+zA_{n-1}}, \dots, M_{A_{n-1}^*+zA_1},M_z\right)M_{\Theta}\right\|\\
		& \leq \left\|p\left(M_{A_1^*+zA_{n-1}}, \dots, M_{A_{n-1}^*+zA_1},M_z\right)\right\| \\& \leq \|p\|_{\infty, \Gamma_n}.
	\end{align*}
	\endgroup
	Thus $\left(M_{B_1+zB_{n-1}^*}, \dots, M_{B_{n-i}+zB_{i}^*},M_z\right)$ is a $\Gamma_n$-contraction such that $M_z$ is a pure isometry. Therefore, $\left(M_{B_1+zB_{n-1}^*}, \dots, M_{B_{n-i}+zB_{i}^*},M_z\right)$ is a pure $\Gamma_n$-isometry and the proof is complete.
	
\end{proof}

\end{document}